\documentclass[11pt]{article}

\usepackage[a4paper,margin=1in]{geometry}

\usepackage{amsmath, amssymb, amsthm, bm}
\usepackage{graphicx}
\usepackage{tikz}
\usetikzlibrary{arrows, positioning, automata, shapes.geometric}

\usepackage[ruled,vlined]{algorithm2e}
\usepackage{subcaption}
\usepackage[numbers]{natbib}

\newtheorem{theorem}{Theorem}

\newtheorem{lemma}{Lemma}
\newtheorem{corollary}{Corollary}

\theoremstyle{definition}
\newtheorem{definition}{Definition}
\newtheorem{example}{Example}
\newtheorem{remark}{Remark}
\newtheorem{question}{Question}
\newtheorem{observation}{Observation}

\title{On graph products and multi-word-representability}

\author{
\begin{tabular}{ccc}
Benny George Kenkireth & Gopalan Sajith & Sreyas Sasidharan \\
\texttt{ben@iitg.ac.in} & \texttt{sajith@iitg.ac.in} & \texttt{sreyas.s@iitg.ac.in}
\end{tabular} \\
\\
Department of Computer Science and Engineering \\
Indian Institute of Technology Guwahati, India
}

\date{}

\begin{document}

\maketitle

\begin{abstract}
The multi-word-representation number of a graph $G$, denoted by $\mu(G)$, is the minimum number of word-representable graphs whose union is $G$. We investigate $\mu(H)$ for the graph $H$ obtained from graphs $G_1$ and $G_2$ via one of six fundamental graph products: the lexicographic, Cartesian, rooted, corona, tensor, and strong products. We prove that both the Cartesian and rooted products satisfy $\mu(H)=\max\{\mu(G_1),\mu(G_2)\}$. For the corona product, we establish the bound $\mu(H)\le \max\{\mu(G_1),\mu(G_2)\}+1$ and show that $\mu(H)=\max\{\mu(G_1),\mu(G_2)\}$ when either $\mu(G_1)>\mu(G_2)$ or $G_2$ admits a covering by $\mu(G_2)$ word-representable graphs, at least one of which is a comparability graph. For the lexicographic product, the general upper bound is shown to be $\mu(G_1)+\mu(G_2)$, and it collapses to $\max\{\mu(G_1),\mu(G_2)\}$ when the cover number of $G_2$ by comparability graphs, $\mathrm{cov}_{\mathrm{comp}}(G_2)$, satisfies $\mathrm{cov}_{\mathrm{comp}}(G_2)\le \max\{\mu(G_1),\mu(G_2)\}$. For the lexicographic product of two minimal non-word-representable graphs, we show that $\mu(H) \le 3$. For the tensor and strong products, we establish logarithmic bounds in terms of the chromatic numbers of $G_1$ and $G_2$; in particular, for the strong product, $\max\{\mu(G_1),\mu(G_2)\}\le \mu(H)\le \max\{\mu(G_1),\mu(G_2)\}+\log_3(\min\{\chi(G_1),\chi(G_2)\})$.

We prove that the lexicographic power $G^{[k]}$ is word-representable if and only if $G$ is a comparability graph. We provide the general upper bound $\mu(G^{[k]})\le \mathrm{cov}_{\mathrm{comp}}(G)$ and, for non-comparability word-representable graphs $G$, the bound $\mu(G^{[k]})\le k$. As an application of lexicographic powers, we study the extremal function $\tau(n)$, the largest integer such that every $n$-vertex graph contains a word-representable induced subgraph of order at least $\tau(n)$. By applying lexicographic powers, we obtain the sublinear upper bound $\tau(n)\le n^{\log_8 6+\epsilon}$ for any $\epsilon>0$. Finally, we resolve the Word-representable Bipartition (WB) problem in the negative for all sufficiently large $n$: we prove that for every $n \geq 433$, there exists a graph of order $n$ that cannot be vertex-partitioned into two word-representable induced subgraphs.

\end{abstract}


\section{Introduction}

The notion of word-representable graphs was introduced by Kitaev and Seif~\cite{perkins} and has since attracted considerable attention (see, e.g.,~\cite{kitaev2017comprehensive,newres,onrepgraphs}). A graph $G=(V, E)$ is \emph{word-representable} if there exists a word $w$ over the alphabet $V$ such that for any distinct $x, y \in V$, the edge $xy \in E$ if and only if the letters $x$ and $y$ alternate in $w$. In this case, the word $w$ is said to \emph{represent} the graph $G$.

Word-representable graphs are of particular interest in graph theory, as they encompass several well-known graph families such as $3$-colorable graphs, subcubic graphs, and comparability graphs~\cite{book}. Moreover, they admit an elegant characterization in terms of \emph{semi-transitive orientations}: a graph is word-representable if and only if it admits a semi-transitive orientation~\cite{book}. An orientation of a graph is \emph{semi-transitive} if it is acyclic and contains no shortcuts; that is, there is no directed path $v_1 \to \cdots \to v_k$ together with a directed edge $v_1 \to v_k$ such that some edge $v_i \to v_r$ is missing for $1 \le i < r \le k$ (see Section~\ref{sec:preliminaries} for a formal definition). Comparability graphs, which admit a transitive orientation, form a natural and important subclass of word-representable graphs.

The concept of \emph{multi-word-representability} was introduced by Kenkireth and Malhotra~\cite{bgkaha} as a natural extension of word-representability. A graph $G$ is \emph{$k$-multi-word-representable} if it can be expressed as the union of at most $k$ word-representable subgraphs, where the union is taken over both their vertex and edge sets. Equivalently, $G$ is $k$-multi-word-representable if there exists a collection of at most $k$ words such that each word represents a subgraph of $G$, and the union of these subgraphs equals $G$. The \emph{multi-word-representation number} of a graph $G$, denoted by $\mu(G)$, is the minimum integer $k$ such that $G$ is $k$-multi-word-representable. By definition, a graph is word-representable if and only if $\mu(G)=1$.

The concept of word-representability is inherently binary, classifying graphs into two distinct categories: word-representable and non-word-representable. In the context of multi-word-representability, the parameter $\mu(G)$ refines this classification by distinguishing non-word-representable graphs according to the minimum number of word-representable subgraphs whose union forms $G$. The motivation for studying this parameter was discussed in~\cite{multi-word-rep}, where $\mu(G)$ was interpreted as measuring the structural deviation of a graph from word-representability, with larger values of $\mu(G)$ indicating a greater degree of such deviation. For instance, a graph $G$ with $\mu(G)=3$ exhibits a greater deviation from word-representability than a graph $H$ with $\mu(H)=2$. Thus, $\mu(G)$ provides a finer complexity classification within the class of non-word-representable graphs. More recently, it was shown in~\cite{HEFTY2026106215} that there exist $n$-vertex graphs with $\mu(G)=\Omega(\log\log n)$, establishing that the multi-word-representation number is unbounded as a function of the order of the graph.

In this work, we investigate the multi-word-representation number $\mu(H)$, where $H$ is the product graph formed from the graphs $G_1$ and $G_2$ via one of six fundamental graph products: the lexicographic, Cartesian, rooted, corona, tensor, or strong products. Our study of the lexicographic, Cartesian, rooted, and corona products is motivated by existing characterizations of word-representability for these graph classes. Evaluating $\mu(H)$ for these four products thus provides a natural extension of their known characterizations. In contrast, the corresponding word-representability characterization problems for the tensor and strong products remain open. Notably, Kitaev and Lozin~\cite{book} explicitly raised the question of characterizing the word-representability of tensor products, noting that similar inquiries could be extended to other graph products. Following this, Choi, Kim, and Kim~\cite{choi2019operations} demonstrated that the tensor and strong products do not necessarily preserve word-representability by providing concrete examples of word-representable graphs whose tensor and strong products yield non-word-representable graphs. With the problem of characterizing the word-representability of tensor and strong products remaining open, we investigate their multi-word-representation numbers and establish general bounds.

Building upon our study of the lexicographic product of graphs, we extend our investigation to the multi-word-representation number for the lexicographic powers of a graph. For a graph $G$ and an integer $k \ge 1$, the $k$-th lexicographic power $G^{[k]}$ represents the repeated application of the lexicographic product to the base graph $G$. In this setting, we establish upper bounds for the multi-word-representation number $\mu(G^{[k]})$ and characterize the word-representability of $G^{[k]}$. As a primary application, we use lexicographic powers to investigate the extremal function $\tau(n)$, defined as the largest integer such that every graph on $n$ vertices contains a word-representable induced subgraph of order at least $\tau(n)$. The function $\tau(n)$ was introduced and evaluated for small orders $n \le 9$ in~\cite{multi-word-rep}. In this paper, we establish an upper bound for $\tau(n)$ showing that it grows sublinearly, while providing a polynomial lower bound for the class of perfect graphs. Crucially, we leverage these asymptotic properties to show that the Word-representable Bipartition (WB) problem, also introduced in~\cite{multi-word-rep}, yields a negative answer for all sufficiently large $n$. The WB problem asks whether the vertex set of any graph can be partitioned into two subsets that induce word-representable subgraphs. While it is known that every graph on $n \le 13$ vertices admits such a bipartition~\cite{multi-word-rep}, we prove that for every integer $n \geq 433$, there exists a graph of order $n$ lacking a word-representable bipartition, thereby bounding the range for which a positive answer can exist.

\paragraph{Our Contributions.}
We summarize our main contributions as follows.

\begin{enumerate}
 \item \textbf{Cartesian and Rooted Products:} We show that if $H$ is either the Cartesian product $G_{1} \square G_{2}$ or the rooted product $G_{1} \diamond G_{2}$ of graphs $G_1$ and $G_2$, then 
    \[
        \mu(H) = \max\{\mu(G_{1}), \mu(G_{2})\}.
    \]

 \item \textbf{Lexicographic Product:} Let $H = G_1 \circ G_2$ be the lexicographic product of graphs $G_1$ and $G_2$. We establish the general lower and upper bounds:
    \[
        \max\{\mu(G_1), \mu(G_2)\} \le \mu(H) \le \mu(G_1) + \mu(G_2).
    \]
    Moreover, we prove that if $\mathrm{cov}_{\mathrm{comp}}(G_2) \le \max\{\mu(G_1), \mu(G_2)\}$, the lower bound is tight, yielding $\mu(H) = \max\{\mu(G_1), \mu(G_2)\}$. In the special case where both $G_1$ and $G_2$ are minimal non-word-representable graphs, we show that $\mu(H) \le 3$.
    
 \item \textbf{Word-representability Characterization of Lexicographic Powers:} For any graph $G$ and integer $k \ge 2$, we show that the $k$-th lexicographic power $G^{[k]}$ is word-representable if and only if $G$ is a comparability graph.

 \item \textbf{Bounds on Multi-word-representation Number for Lexicographic Powers:} For any graph $G$ and an integer $k \ge 2$, we show that
    \[
        \mu(G^{[k]}) \le \mathrm{cov}_{\mathrm{comp}}(G).
    \]
    Furthermore, if $G$ is word-representable but not a comparability graph, then $\mu(G^{[k]}) \le k$.

 \item \textbf{Bounds on the Function $\tau(n)$:} For any integer $k \ge 1$, we show that $\tau(8^{k}) \le 6^{k}$, which implies that
    \[
        \tau(n) \le n^{\log_8 6 + \epsilon},
    \]
    for any $\epsilon > 0$ and sufficiently large $n$. Conversely, when restricted to the class of perfect graphs on $n$ vertices, we establish the polynomial lower bound:
    \[
        \tau_{\mathrm{perf}}(n) \ge n^{1/2}.
    \]

 \item \textbf{Negative Answer to the WB Problem for Large $n$:} We prove that the WB problem has a negative answer for all $n \geq 433$. Specifically, we show that for every integer $n \geq 433$, there exists a graph of order $n$ that does not admit a word-representable bipartition.
    
\item \textbf{Corona Product:} Let $H = G_1 \odot G_2$ be the corona product of graphs $G_1$ and $G_2$. We establish the general bounds:
    \[
        \max\{\mu(G_1), \mu(G_2)\} \le \mu(H) \le \max\{\mu(G_1), \mu(G_2)\} + 1.
    \]
    Moreover, we prove that $\mu(H) = \max\{\mu(G_1), \mu(G_2)\}$ if either $\mu(G_1) > \mu(G_2)$ or $G_2$ admits a covering by $\mu(G_2)$ word-representable graphs, at least one of which is a comparability graph.
    
 \item \textbf{Tensor Product:} Let $H = G_1 \times G_2$ be the tensor product of graphs $G_1$ and $G_2$. We show that
    \[
        \mu(H) \le \log_3 \bigl(\min\{\chi(G_1), \chi(G_2)\}\bigr).
    \]
    As an immediate consequence, we observe that if at least one of the graphs ($G_{1}$ or $G_{2}$) is $3$-colorable, their tensor product is guaranteed to be word-representable.

 \item \textbf{Strong Product:} Let $H = G_1 \boxtimes G_2$ be the strong product of graphs $G_1$ and $G_2$. We establish the bounds
    \[
        \max\{\mu(G_1), \mu(G_2)\} \le \mu(H) \le \max\{\mu(G_1), \mu(G_2)\} + \log_3 \bigl(\min\{\chi(G_1), \chi(G_2)\}\bigr).
    \]
\end{enumerate}

Here, the parameter $\mathrm{cov}_{\mathrm{comp}}(G)$ denotes the cover number of the graph $G$ by comparability graphs, defined as the minimum number of comparability graphs needed to cover the edge set of $G$ (see Definition~\ref{cov-num-defn}).

The remainder of this paper is organized as follows. Section~\ref{sec:preliminaries} introduces the necessary preliminaries and foundational definitions. Sections~\ref{sec:cartesian} and~\ref{sec:rooted} determine the exact multi-word-representation numbers for the Cartesian and rooted products of two arbitrary graphs, respectively. Section~\ref{sec:lexicographic} establishes general bounds on the multi-word-representation number of the lexicographic product and identifies conditions under which exact values can be computed. This is extended in Section~\ref{sec:lexpower_and_mul}, which derives bounds for the lexicographic powers of a graph. Next, Section~\ref{sec:tau:n} explores the extremal function $\tau(n)$, establishing a sublinear upper bound via lexicographic powers alongside general and restricted lower bounds. Finally, we apply these asymptotic results to resolve the WB problem in the negative for all sufficiently large $n$. Section~\ref{sec:corona} presents general bounds for the corona product of two graphs and details the instances where exact values are achieved. Sections~\ref{sec:tensor} and~\ref{sec:strong} then establish bounds for the tensor and strong products, respectively. Finally, Section~\ref{sec:conclusion} concludes the paper and outlines several open questions for future research.

\section{Preliminaries}
\label{sec:preliminaries}

All graphs considered in this paper are simple and undirected. For a graph $G$, we denote its vertex set by $V(G)$ and its edge set by $E(G)$. For $A \subseteq V(G)$, the induced subgraph of $G$ on $A$ is denoted by $G[A]$. The chromatic number, the independence number, and the clique number of a graph $G$ are denoted by $\chi(G)$, $\alpha(G)$, and $\omega(G)$, respectively. Unless otherwise specified, all logarithms considered are to the base $2$.

\begin{definition}
Let $w$ be a word over the alphabet $\Sigma$. For any pair of distinct letters $a, b \in \Sigma$, let $w|_{ab}$ denote the word $w$ restricted to the copies of $a$ and $b$ in the order in which they appear. The letters $a$ and $b$ \emph{alternate} in $w$ if the substrings $aa$ and $bb$ are absent in $w|_{ab}$. 
\end{definition}

\begin{example}
Consider the word $w = 1213$. We examine pairs of letters to determine whether they alternate in $w$. The letters $1$ and $2$ alternate, as $w|_{12} = 121$ does not contain the substrings $11$ or $22$. The letters $2$ and $3$ alternate, as $w|_{23} = 23$ does not contain the substrings $22$ or $33$. Conversely, the letters $1$ and $3$ do not alternate, as $w|_{13} = 113$ contains the substring $11$.
\end{example}

\begin{definition}
\label{word-rep-dfn}
A graph $G$ is \emph{word-representable} if there exists a word $w$ over the alphabet $V(G)$ such that for any distinct pair of vertices $x, y \in V(G)$, $(x, y) \in E(G)$ if and only if the letters $x$ and $y$ alternate in $w$. The word $w$ is said to \emph{represent} the graph $G$ and is called a \emph{word-representant} of $G$.
\end{definition}

\begin{example}
\label{example-word-rep-graph}
We illustrate Definition \ref{word-rep-dfn} using two fundamental graph classes:
\begin{enumerate}
    \item Let $G$ be an independent set on $n$ vertices $\{v_1, v_2, \dots, v_n\}$. The word $w = v_1 v_1 v_2 v_2 \dots v_n v_n$ represents $G$. Indeed, for any pair of distinct vertices $v_i$ and $v_j$ (with $i < j$), the restricted word is $w|_{v_i v_j} = v_i v_i v_j v_j$. Since the substrings $v_i v_i$ and $v_j v_j$ are present, the letters do not alternate in $w$, meaning no two vertices are adjacent.
    
    \item Let $G$ be a complete graph (clique) on $n$ vertices $\{v_1, v_2, \dots, v_n\}$. The word $w = v_1 v_2 \dots v_n$, containing exactly one occurrence of each vertex, represents $G$. For any distinct pair $v_i$ and $v_j$ (with $i < j$), the restriction yields $w|_{v_i v_j} = v_i v_j$. Because the substrings $v_i v_i$ and $v_j v_j$ are absent in $w|_{v_i v_j}$, the letters alternate in $w$. Thus, every pair of distinct vertices is adjacent.
\end{enumerate}
The smallest non-word-representable graph is the wheel graph $W_5$ on six vertices. 
\end{example}

Comparability graphs form a subclass of word-representable graphs. 

\begin{definition}
A graph $G = (V, E)$ is a \emph{comparability graph} if it admits a transitive orientation of its edges; that is, an orientation of $E$ such that for any vertices $a, b, c \in V$, the directed edges $a \rightarrow b$ and $b \rightarrow c$ imply the existence of the directed edge $a \rightarrow c$. A graph that does not admit such an orientation is a \emph{non-comparability graph}.
\end{definition}

The cliques and independent sets discussed in Example \ref{example-word-rep-graph} are also basic examples of comparability graphs, whereas the cycle graph $C_5$ is a well-known non-comparability graph \cite{Gallai1967TransitivOG}. 

One notable structural relationship between word-representable graphs and comparability graphs is highlighted by the following characterization for graphs containing a universal vertex.

\begin{theorem}[\cite{onrepgraphs}]
\label{comp-all-adj-wr-character}
Let $G$ be a graph on $n$ vertices containing a vertex $x \in V(G)$ of degree $n-1$. Let $H = G[V(G) \setminus \{x\}]$ be the induced subgraph obtained by removing $x$. Then, $G$ is word-representable if and only if $H$ is a comparability graph.
\end{theorem}

It is a well-known fact that both the class of word-representable graphs and the class of comparability graphs are hereditary. Consequently, if a graph contains a non-word-representable (or a non-comparability) induced subgraph, then the graph itself is non-word-representable (or a non-comparability graph, respectively).

\begin{definition}
A graph $G$ is a \emph{minimal non-word-representable graph} if $G$ is not word-representable, but all its proper induced subgraphs are word-representable.
\end{definition}

The wheel graph $W_5$ on six vertices is also an example of a minimal non-word-representable graph.
Word-representable graphs admit a characterization in terms of semi-transitive orientations.

\begin{definition}
\label{semi-trans-defn}
A graph $G$ is \emph{semi-transitive} if it admits an acyclic orientation such that for every directed path $v_1 \rightarrow \dots \rightarrow v_k$, either $v_1 \not\rightarrow v_k$, or if $v_1 \rightarrow v_k$, then $v_i \rightarrow v_j$ exists for all $1 \le i < j \le k$. If $v_1 \rightarrow v_k$ is present while some $v_i \rightarrow v_j$ with $i<j$ is missing, this configuration forms a \emph{shortcut}~\cite{book}.
\end{definition}

\begin{theorem}[\cite{book}]
\label{wrg=semi}
A graph $G$ is word-representable if and only if it is semi-transitive.
\end{theorem}

We next define the union of a collection of graphs.

\begin{definition}
For graphs $G_1,\dots,G_k$, their \emph{union}, denoted by $\bigcup_{i=1}^k G_i$, is the graph $G$ where
\[
V(G)=\bigcup_{i=1}^k V(G_i)
\quad \text{and} \quad
E(G)=\bigcup_{i=1}^k E(G_i).
\]
\end{definition}

When these graphs have no vertices in common, their union preserves word-representability.

\begin{theorem}[\cite{bgkaha}]
\label{unionofwrgiswrg}
Let $G_1, \dots, G_n$ be word-representable graphs with disjoint vertex sets. Then their union $\bigcup_{i=1}^n G_i$ is word-representable.
\end{theorem}

In terms of connected components, the property is characterized as follows.

\begin{theorem}[\cite{book}]
\label{word-rep-components-thm}
A graph $G$ is word-representable if and only if every connected component of $G$ is word-representable.
\end{theorem}

The concept of multi-word-representability generalizes word-representability via graph coverings.

\begin{definition}
A collection of graphs $G_1, G_2, \dots, G_k$ is said to \emph{cover} a graph $G$ if
\[
G = \bigcup_{i=1}^k G_i.
\]
\end{definition}

\begin{definition}
\label{defn-covering-word-rep}
A graph $G$ is said to admit a \emph{covering by word-representable graphs} if there exist word-representable graphs $G_1, G_2, \dots, G_k$ that cover $G$.
\end{definition}

\begin{definition}
Let $k \ge 1$ be an integer. A graph $G$ is called \emph{$k$-multi-word-representable} if $G$ admits a covering by at most $k$ word-representable graphs.
\end{definition}

This classification naturally introduces the following parameter.

\begin{definition}
The \emph{multi-word-representation number} of a graph $G$, denoted by $\mu(G)$, is the smallest positive integer $k$ such that $G$ is $k$-multi-word-representable.
\end{definition}
This parameter is non-decreasing under the induced subgraph relation.

\begin{remark}
\label{monotonic-mu-hereditary-remark}
For every integer $k \geq 1$, the class of $k$-multi-word-representable graphs is hereditary. Indeed, let $H$ be an induced subgraph of a $k$-multi-word-representable graph $G$, where $G = \bigcup_{i=1}^k G_i$ and each $G_i$ is word-representable. Then
\[
H = \bigcup_{i=1}^k G_i[V(H) \cap V(G_i)].
\]
Since the class of word-representable graphs is hereditary, each induced subgraph $G_i[V(H) \cap V(G_i)]$ is word-representable. Hence, $H$ is a union of at most $k$ word-representable graphs, and is therefore $k$-multi-word-representable.

In particular, if $H$ is an induced subgraph of $G$, then $\mu(H) \le \mu(G)$.
\end{remark}

Based on the fact that every $3$-colorable graph is word-representable \cite{book}, the following bound holds.
\begin{theorem}[\cite{bgkaha}]
\label{mu-log-chi-bound}
Let $G$ be a graph. Then
\[
\mu(G) \le \lceil \log_{3}\chi(G) \rceil.
\]
\end{theorem}

We recall the following parameters from \cite{multi-word-rep} regarding maximum size word-representable induced subgraphs.

\begin{definition}[\cite{multi-word-rep}]
\label{def:representable-set}
Let $G$ be a graph and let $S \subseteq V(G)$.
\begin{itemize}
    \item The set $S$ is \emph{representable} if the induced subgraph $G[S]$ is word-representable.
    \item A representable set $S$ is \emph{maximum} if it has maximum cardinality among all representable subsets of $V(G)$.
\end{itemize}
The cardinality of a maximum representable set of $G$ is denoted by $\eta(G)$.
\end{definition}

\begin{definition}[\cite{multi-word-rep}]
\label{def:tau-function}
For a positive integer $n$, the function $\tau(n)$ is defined as
\[
\tau(n) = \min\{\eta(G) \mid G \text{ is a graph on } n \text{ vertices}\}.
\]
Equivalently, $\tau(n)$ is the largest integer such that every graph on $n$ vertices contains a word-representable induced subgraph on at least $\tau(n)$ vertices.
\end{definition}

Finally, we recall the Word-representable Bipartition (WB) problem introduced in \cite{multi-word-rep}.

\begin{definition}
A graph $G$ \emph{admits a word-representable bipartition} if its vertex set $V(G)$ can be partitioned into two subsets $V_1$ and $V_2$ such that both induced subgraphs $G[V_1]$ and $G[V_2]$ are word-representable.
\end{definition}

\begin{remark}[The WB Problem \cite{multi-word-rep}]
\label{rem:wb_problem}
The Word-representable Bipartition (WB) problem asks whether every graph on $n$ vertices admits a word-representable bipartition for all $n \geq 1$.
\end{remark}

We now define the notion of the cover number of a graph by comparability graphs.

\begin{definition}
\label{cov-num-defn}
The \emph{cover number of a graph $G$ by comparability graphs}, denoted by $\mathrm{cov}_{\mathrm{comp}}(G)$, is the minimum integer $k$ such that there exist comparability graphs $G_1, G_2, \dots, G_k$ satisfying
\[
E(G) = E(G_1) \cup E(G_2) \cup \cdots \cup E(G_k).
\]

\end{definition}

For a comparability graph $G$, $\mathrm{cov}_{\mathrm{comp}}(G) = 1$, and for a non-comparability graph $G$, $\mathrm{cov}_{\mathrm{comp}}(G) \geq 2$.

\begin{remark}
\label{spanning-cover-remark}
In certain proofs, we will assume without loss of generality that the graphs $G_1, \dots, G_k$ forming a covering of $G$ (whether by word-representable graphs or comparability graphs) are spanning subgraphs of $G$; that is, $V(G_i) = V(G)$ for all $1 \le i \le k$. If a covering graph $G_i$ is not spanning, the missing vertices $V(G) \setminus V(G_i)$ can be added to $G_i$ as isolated vertices. Since adding isolated vertices preserves both word-representability and the property of being a comparability graph, this operation allows us to assume the covering consists entirely of spanning subgraphs without altering the edge union, the multi-word-representation number, or the cover number of $G$ by comparability graphs.
\end{remark}

To investigate the structural behavior of the multi-word-representation number under graph products, this paper considers six distinct graph products. For four of these operations, namely, the Cartesian, rooted, lexicographic, and corona products, the exact conditions under which the resulting product graph is word-representable are already known. Investigating their multi-word-representation numbers serves as a natural extension of the existing characterizations. 

Conversely, for the tensor and strong products, characterizing word-representability remains an open problem. Specifically for the tensor product, this investigation was posed as an open question by Kitaev in~\cite[Chapter~7]{book}. While these characterizations remain open, we advance the study of these operations by establishing an upper bound on the multi-word-representation number for the tensor product, along with both upper and lower bounds for the multi-word-representation number of the strong product.

We begin by recalling the definition of the Cartesian product of two graphs below.

\begin{definition}
Let $G_{1}$ and $G_{2}$ be graphs. The \emph{Cartesian product} $H = G_{1} \square G_{2}$ is the graph defined by
\[
V(H) = V(G_{1}) \times V(G_{2}),
\]
where two vertices $(u_{1}, v_{1})$ and $(u_{2}, v_{2})$ in $V(H)$ are adjacent if and only if either:
\begin{enumerate}
    \item $u_{1} = u_{2}$ and $v_{1}v_{2} \in E(G_{2})$, or
    \item $v_{1} = v_{2}$ and $u_{1}u_{2} \in E(G_{1})$.
\end{enumerate}
\end{definition}

\begin{theorem}[\cite{book}]
\label{cartesian-word-rep-thm}
Let $G_{1}$ and $G_{2}$ be word-representable graphs. Then the Cartesian product $G_{1} \square G_{2}$ is word-representable.
\end{theorem}

We recall the definition of the rooted product of two graphs below.

\begin{definition}
Let $G_{1}$ and $G_{2}$ be graphs, and let $v_{0} \in V(G_{2})$ be a distinguished vertex called the root of $G_{2}$. The \emph{rooted product} $H = G_{1} \diamond G_{2}$ is the graph with vertex set 
\[
V(H) = V(G_{1}) \times V(G_{2}),
\]
where two vertices $(u_1, v_1)$ and $(u_2, v_2)$ are adjacent if and only if either:
\begin{enumerate}
    \item $u_1 = u_2$ and $v_1v_2 \in E(G_2)$, or
    \item $v_1 = v_2 = v_{0}$ and $u_1u_2 \in E(G_1)$.
\end{enumerate}
\end{definition}

\begin{theorem}[\cite{book}]
\label{rooted-char-word-rep-thm}
Let $G_{1}$ and $G_{2}$ be word-representable graphs. Then the rooted product $G_{1} \diamond G_{2}$ is word-representable.
\end{theorem}

We recall the definition of the lexicographic product of two graphs below.

\begin{definition}
Let $G_1$ and $G_2$ be two graphs. The \emph{lexicographic product} $H = G_1 \circ G_2$ is the graph with vertex set $V(H) = V(G_1) \times V(G_2)$, where two vertices $(u,v)$ and $(x,y)$ are adjacent in $H$ if and only if either:
\begin{enumerate}
    \item $ux \in E(G_1)$, or
    \item $u = x$ and $vy \in E(G_2)$.
\end{enumerate}
\end{definition}

Alternatively, the lexicographic product can be visualized by replacing each vertex $v_i \in V(G_1)$ with a copy of $G_2$, denoted as the \emph{supervertex} 
\[
V_i = \{ (v_i, u) \mid u \in V(G_2) \} \subseteq V(H).
\] 
The subgraph of $H$ induced by $V_i$ is isomorphic to $G_2$. For each edge $v_i v_j \in E(G_1)$, every vertex in $V_i$ is adjacent to every vertex in $V_j$, whereas if $v_i v_j \notin E(G_1)$, there are no edges between $V_i$ and $V_j$. For each supervertex $V_i$, we refer to $v_i \in V(G_1)$ as its \emph{corresponding vertex} in $G_1$; conversely, we refer to $V_i$ as the \emph{corresponding supervertex} of $v_i$.

\begin{theorem}[\cite{dwary2025characterization}]
\label{lex-prod-char-word-re}
Let $G_1$ and $G_2$ be graphs. The lexicographic product $G_1 \circ G_2$ is word-representable if and only if $G_1$ is word-representable and $G_2$ is a comparability graph.
\end{theorem}

\begin{remark}
We note that Theorem~\ref{lex-prod-char-word-re}, as formulated in~\cite{dwary2025characterization}, requires the additional condition that $G_1$ is a non-empty graph (i.e., $G_1$ contains at least one edge). Indeed, if $G_1$ is an edgeless graph on $n$ vertices ($n \ge 1$) and $G_2$ is the cycle graph $C_5$, then $G_2$ is word-representable but not a comparability graph. However, the lexicographic product $G_1 \circ G_2$ is simply the disjoint union of $n$ copies of $C_5$. By Theorem~\ref{unionofwrgiswrg}, the disjoint union of word-representable graphs remains word-representable, which provides a counterexample to the necessity of $G_2$ being a comparability graph when $G_1$ is edgeless. Consequently, the condition that $G_1$ contains at least one edge is necessary for the theorem to hold.
\end{remark}

\begin{theorem}[\cite{dwary2025characterization}]
\label{lex-prod-char-compara}
Let $G_1$ and $G_2$ be graphs. The lexicographic product $G_1 \circ G_2$ is a comparability graph if and only if $G_1$ and $G_2$ are comparability graphs.
\end{theorem}

\begin{definition}
Let $G$ be a graph and $k$ a positive integer. The \emph{$k$-th lexicographic power} of $G$, denoted $G^{[k]}$, is defined recursively as
\[
G^{[1]} = G, \quad G^{[k]} = G^{[k-1]} \circ G \text{ for } k \ge 2,
\]
where $\circ$ denotes the lexicographic product.
\end{definition}

\begin{remark}
Since the lexicographic product is associative, $G^{[k]}$ is well-defined. Equivalently, one may write $G^{[k]} = G \circ G^{[k-1]}$.  
\end{remark}

\begin{definition}
Let $G_1$ and $G_2$ be graphs, and let $H = G_1 \circ G_2$ be their lexicographic product. For any subgraph $G_{1}'$ of $G_{1}$, the \emph{lexicographic map} of $G_1'$ in $H$, denoted by $\mathcal{L}_{H}(G_1')$, is the spanning subgraph of $H$ with vertex set $V(\mathcal{L}_{H}(G_1')) = V(H)$ and edge set 
\[
E(\mathcal{L}_{H}(G_1')) = \{ (u,x)(v,y) \mid uv \in E(G_1') \text{ and } x,y \in V(G_2) \}.
\]
\end{definition}

\begin{remark}
\label{remark-lex-map}   
Structurally, the lexicographic map $\mathcal{L}_{H}(G_1)$ is isomorphic to the lexicographic product $G_1 \circ \overline{K_n}$, where $\overline{K_n}$ is the empty (edgeless) graph on $n = |V(G_2)|$ vertices. 
\end{remark}

The following theorem follows from Theorem~\ref{lex-prod-char-word-re}, Theorem~\ref{lex-prod-char-compara}, Remark~\ref{remark-lex-map}, and the fact that the complement graph of $K_{n}$, for any integer $n$, is a comparability graph.

\begin{theorem}
\label{thm:lexmap-merged}
Let $H' = \mathcal{L}_H(G_1')$. Then $H'$ is word-representable (resp. a comparability graph) if and only if $G_1'$ is word-representable (resp. a comparability graph).
\end{theorem}

\begin{definition}
\label{defn-special}
Let $G_{1}$ and $G_{2}$ be graphs, and let $H = G_{1} \circ G_{2}$ denote their lexicographic product. A subgraph $H'$ of $H$ is called a \emph{special subgraph} if it can be expressed as $H' = G_{1}' \circ G_{2}'$, where $G_{1}'$ is a word-representable subgraph of $G_{1}$ and $G_{2}'$ is a spanning comparability subgraph of $G_{2}$.
\end{definition}

\begin{remark}
If $G_{1}$ is word-representable and $G_{2}$ is a comparability graph, then $H = G_{1} \circ G_{2}$ is itself a special graph.
\end{remark}

\begin{remark}
\label{lex-map-special-connection-remark}
Let $H = G_1 \circ G_2$, let $G_1'$ be a word-representable spanning subgraph of $G_1$, and let $G_2'$ be a spanning comparability subgraph of $G_2$. The special subgraph $H' = G_1' \circ G_2'$ is precisely the union of the lexicographic map $\mathcal{L}_H(G_1')$ and the disjoint union of copies of $G_2'$ across all supervertices of $H$.
\end{remark}

As an application of Theorem~\ref{lex-prod-char-word-re} and Theorem~\ref{lex-prod-char-compara}, we obtain the following observation on special subgraphs of lexicographic products.

\begin{observation}
\label{special-word-rep-comp-char}
Let $G_{1}$ and $G_{2}$ be graphs, and let $H = G_{1} \circ G_{2}$. Any special subgraph $H_{s}$ of $H$ is word-representable. Moreover, if $H_{s} = G_{1}' \circ G_{2}'$, where both $G_{1}'$ and $G_{2}'$ are comparability graphs, then $H_{s}$ is a comparability graph.
\end{observation}

We recall the definition of the corona product of two graphs below.

\begin{definition}
Let $G_{1}$ and $G_{2}$ be graphs. The \emph{corona product} $H = G_{1} \odot G_{2}$ is the graph obtained by taking one copy of $G_{1}$ and, for each vertex $v \in V(G_{1})$, a copy of $G_{2}$, denoted by $G_{2}^{v}$, and joining $v$ to every vertex of $G_{2}^{v}$. Formally, the vertex set of $H$ is
\[
V(H) = V(G_{1}) \cup \bigcup_{v \in V(G_{1})} V(G_{2}^{v}),
\]
and the edge set of $H$ is
\[
E(H) = E(G_{1}) \cup \bigcup_{v \in V(G_{1})} E(G_{2}^{v}) \cup \bigcup_{v \in V(G_{1})} \{vu : u \in V(G_{2}^{v})\}.
\]
Each $G_{2}^{v}$ is called the \emph{copy of $G_{2}$ attached to $v$}.
\end{definition}

\begin{theorem}[\cite{Corona-product}]
\label{corona-prod-chara-word}
Let $G_1$ and $G_2$ be graphs. The corona product $G_1 \odot G_2$ is word-representable if and only if $G_1$ is word-representable and $G_2$ is a comparability graph.
\end{theorem}

We now recall the definition of the tensor product of graphs.

\begin{definition}
\label{tensor-prod-defn}
Let $G_{1}$ and $G_{2}$ be graphs. The \emph{tensor product} $H = G_{1} \times G_{2}$ is the graph with vertex set $V(H) = V(G_{1}) \times V(G_{2})$, where two vertices $(u_{1}, v_{1})$ and $(u_{2}, v_{2})$ are adjacent in $H$ if and only if $u_{1}u_{2} \in E(G_{1})$ and $v_{1}v_{2} \in E(G_{2})$.
\end{definition}

The chromatic number of the tensor product satisfies the well-known upper bound $\chi(G_{1} \times G_{2}) \le \min \{\chi(G_{1}), \chi(G_{2})\}$, as established by Hedetniemi~\cite{Hedetniemi1967HomomorphismsOG}.

\begin{definition}
\label{def-strong-product}
Let $G_{1}$ and $G_{2}$ be graphs. The \emph{strong product} $H = G_{1} \boxtimes G_{2}$ is the graph with vertex set $V(H) = V(G_{1}) \times V(G_{2})$, where two distinct vertices $(u_{1}, v_{1})$ and $(u_{2}, v_{2})$ are adjacent in $H$ if and only if one of the following conditions holds:
\begin{enumerate}
    \item $u_{1} = u_{2}$ and $v_{1}v_{2} \in E(G_{2})$,
    \item $v_{1} = v_{2}$ and $u_{1}u_{2} \in E(G_{1})$, or
    \item $u_{1}u_{2} \in E(G_{1})$ and $v_{1}v_{2} \in E(G_{2})$.
\end{enumerate}
\end{definition}

Equivalently, as established by Sabidussi~\cite{sabidussi1959graph}, the edge set of the strong product can be expressed directly as the union of the edge sets of the Cartesian product and the tensor product:
\[
E(G_{1} \boxtimes G_{2}) = E(G_{1} \square G_{2}) \cup E(G_{1} \times G_{2}).
\]

\section{Cartesian product of graphs and multi-word-representation number}
\label{sec:cartesian}

In this section, we determine the exact value of the multi-word-representation number for the Cartesian product of two arbitrary graphs.

\begin{theorem}
\label{thm:cartesian-multi-rep}
Let $G_{1}$ and $G_{2}$ be graphs, and let $H = G_{1} \square G_{2}$. Then 
\[
\mu(H) = \max\{\mu(G_{1}), \mu(G_{2})\}.
\]
\end{theorem}

\begin{proof}
We first establish the lower bound. For any fixed vertex $v \in V(G_{2})$, the subgraph of $H$ induced by $V(G_{1}) \times \{v\}$ is isomorphic to $G_{1}$. By Remark~\ref{monotonic-mu-hereditary-remark}, it follows that $\mu(G_{1}) \leq \mu(H)$. Similarly, fixing a vertex $u \in V(G_{1})$ yields an induced subgraph isomorphic to $G_{2}$, implying $\mu(G_{2}) \leq \mu(H)$. Consequently,
\[
\mu(H) \geq \max\{\mu(G_{1}), \mu(G_{2})\}.
\]

Next, we establish the upper bound. Let $\mu(G_{1}) = k_{1}$ and $\mu(G_{2}) = k_{2}$, and let $k = \max\{k_{1}, k_{2}\}$. Then $G_{1}$ can be expressed as the union of $k_{1}$ word-representable subgraphs $A_{1}, \dots, A_{k_{1}}$, and $G_{2}$ as the union of $k_{2}$ word-representable subgraphs $B_{1}, \dots, B_{k_{2}}$.
We may assume, without loss of generality, that each $A_i$ and $B_j$ is spanning, by adding isolated vertices if necessary. For indices satisfying $k_{1} < i \leq k$ and $k_{2} < j \leq k$, define $A_{i} = (V(G_{1}), \emptyset)$ and $B_{j} = (V(G_{2}), \emptyset)$, each of which is an edgeless graph and hence trivially word-representable. Thus,
\[
G_{1} = \bigcup_{i=1}^{k} A_{i}, \quad G_{2} = \bigcup_{j=1}^{k} B_{j}.
\]

For each $i \in \{1, 2, \ldots, k\}$, define the graph $D_{i} = A_{i} \square B_{i}$. Since $A_{i}$ and $B_{i}$ are word-representable, each $D_{i}$ is word-representable by Theorem~\ref{cartesian-word-rep-thm}. Moreover, $V(D_{i}) = V(G_{1}) \times V(G_{2}) = V(H)$, so each $D_{i}$ is a spanning subgraph of $H$.

To show that $H = \bigcup_{i=1}^{k} D_i$, let $e = ((u_{1}, v_{1}), (u_{2}, v_{2})) \in E(H)$. By the definition of the Cartesian product, either:
\begin{enumerate}
    \item $v_{1} = v_{2}$ and $u_{1} u_{2} \in E(G_{1})$. Since $G_{1} = \bigcup_{i=1}^{k} A_{i}$, there exists $i \in \{1, \dots, k\}$ such that $u_{1} u_{2} \in E(A_{i})$. As $v_{1} \in V(B_{i})$, we obtain $e \in E(A_{i} \square B_{i}) = E(D_{i})$.
    
     \item $u_{1} = u_{2}$ and $v_{1} v_{2} \in E(G_{2})$. Since $G_{2} = \bigcup_{i=1}^{k} B_{i}$, there exists $i \in \{1, \dots, k\}$ such that $v_{1} v_{2} \in E(B_{i})$. As $u_{1} \in V(A_{i})$, we obtain $e \in E(A_{i} \square B_{i}) = E(D_{i})$.
\end{enumerate}
Hence every edge of $H$ belongs to some $D_i$, and since each $D_i$ is a spanning subgraph of $H$, we obtain
\[
H = \bigcup_{i=1}^{k} D_i.
\]
Thus, $\mu(H) \leq k = \max\{\mu(G_{1}), \mu(G_{2})\}$. Combining both bounds completes the proof.

\end{proof}

\section{Rooted product of graphs and multi-word-representation number}
\label{sec:rooted}

In this section, we determine the exact value of the multi-word-representation number for the rooted product of two graphs.

\begin{theorem}
\label{thm:rooted-multi-rep}
Let $G_{1}$ and $G_{2}$ be graphs, where $G_{2}$ is rooted at a designated vertex $v_{0}$, and let $H = G_{1} \diamond G_{2}$. Then
\[
\mu(H) = \max\{\mu(G_{1}), \mu(G_{2})\}.
\]
\end{theorem}

\begin{proof}
We first establish the lower bound. Considering the root vertex $v_{0} \in V(G_{2})$, the subgraph of $H$ induced by the vertex set $V(G_{1}) \times \{v_{0}\}$ is isomorphic to $G_{1}$. By Remark~\ref{monotonic-mu-hereditary-remark}, it follows that $\mu(G_{1}) \leq \mu(H)$. Similarly, selecting any fixed vertex $u \in V(G_{1})$ yields an induced subgraph isomorphic to $G_{2}$ on the vertex set $\{u\} \times V(G_{2})$, implying $\mu(G_{2}) \leq \mu(H)$. Consequently,
\[
\mu(H) \geq \max\{\mu(G_{1}), \mu(G_{2})\}.
\]

Next, we establish the upper bound. Let $\mu(G_{1}) = k_{1}$ and $\mu(G_{2}) = k_{2}$, and let $k = \max\{k_{1}, k_{2}\}$. Then $G_{1}$ can be expressed as the union of $k_{1}$ word-representable subgraphs $A_{1}, \dots, A_{k_{1}}$, and $G_{2}$ as the union of $k_{2}$ word-representable subgraphs $B_{1}, \dots, B_{k_{2}}$.
We may assume, without loss of generality, that each $A_i$ and $B_j$ is spanning, by adding isolated vertices if necessary. For indices satisfying $k_{1} < i \leq k$ and $k_{2} < j \leq k$, define $A_{i} = (V(G_{1}), \emptyset)$ and $B_{j} = (V(G_{2}), \emptyset)$, each of which is an edgeless graph and hence trivially word-representable. Thus,
\[
G_{1} = \bigcup_{i=1}^{k} A_{i}, \quad G_{2} = \bigcup_{j=1}^{k} B_{j}.
\]

For each $i \in \{1, 2, \ldots, k\}$, define the graph $D_{i} = A_{i} \diamond B_{i}$, where each $B_i$ is rooted at the same vertex $v_0$ as $G_2$. Since $A_{i}$ and $B_{i}$ are word-representable, each $D_{i}$ is word-representable by Theorem~\ref{rooted-char-word-rep-thm}. Moreover, $V(D_{i}) = V(H)$, so each $D_{i}$ is a spanning subgraph of $H$.

To show that $H = \bigcup_{i=1}^{k} D_i$, let $e = ((u_{1}, v_{1}), (u_{2}, v_{2})) \in E(H)$. By the definition of the rooted product, either:
\begin{enumerate}
    \item $v_{1} = v_{2} = v_{0}$ and $u_{1} u_{2} \in E(G_{1})$. Since $G_{1} = \bigcup_{i=1}^{k} A_{i}$, there exists $i \in \{1, \dots, k\}$ such that $u_{1} u_{2} \in E(A_{i})$. As $v_{0} \in V(B_{i})$, we obtain $e \in E(A_{i} \diamond B_{i}) = E(D_{i})$.
    \item $u_{1} = u_{2}$ and $v_{1} v_{2} \in E(G_{2})$. Since $G_{2} = \bigcup_{i=1}^{k} B_{i}$, there exists $i \in \{1, \dots, k\}$ such that $v_{1} v_{2} \in E(B_{i})$. As $u_{1} \in V(A_{i})$, we obtain $e \in E(A_{i} \diamond B_{i}) = E(D_{i})$.
\end{enumerate}
Hence every edge of $H$ belongs to some $D_i$, and since each $D_i$ is a spanning subgraph of $H$, we obtain
\[
H = \bigcup_{i=1}^{k} D_i.
\]
Thus, $\mu(H) \leq k = \max\{\mu(G_{1}), \mu(G_{2})\}$. Combining both bounds completes the proof.

\end{proof}

\section{Lexicographic product of graphs and multi-word-representation number}
\label{sec:lexicographic}

In this section, we investigate the multi-word-representation number of the lexicographic product of graphs. Theorem~\ref{gen-thm-multi-repno} establishes a general upper bound for $\mu(G_1 \circ G_2)$ and provides an exact evaluation under a specified structural condition. Structurally, the characterization of word-representable graphs dictates that the base case $\mu(G_1 \circ G_2) = 1$ depends entirely on whether $G_2$ is a comparability graph. This dependency underscores the relevance of the cover number by comparability graphs, $\mathrm{cov}_{\mathrm{comp}}(G_2)$, as a critical parameter for analyzing the behavior of $\mu(G_1 \circ G_2)$. Guided by this parameter, we partition our analysis into two natural structural regimes: the first where $\mathrm{cov}_{\mathrm{comp}}(G_2) \le \max\{\mu(G_1), \mu(G_2)\}$, which yields an exact evaluation of $\mu(G_1 \circ G_2) = \max\{\mu(G_1), \mu(G_2)\}$, and the second where $\mathrm{cov}_{\mathrm{comp}}(G_2) > \max\{\mu(G_1), \mu(G_2)\}$, wherein a potential gap between the bounds persists.

\begin{theorem}
\label{gen-thm-multi-repno}
Let $G_1$ and $G_2$ be graphs, and let $H = G_1 \circ G_2$. Then
\[
\max\{\mu(G_1), \mu(G_2)\} \le \mu(H) \le \mu(G_1) + \mu(G_2).
\]
Moreover, if $\mathrm{cov}_{\mathrm{comp}}(G_2) \le \max\{\mu(G_1), \mu(G_2)\}$, then
\[
\mu(H) = \max\{\mu(G_1), \mu(G_2)\}.
\]
\end{theorem}

\begin{proof}
Let $\mu(G_1)=k_1$ and $\mu(G_2)=k_2$. By Definition~\ref{defn-covering-word-rep}, there exist word-representable subgraphs $A_1, \dots, A_{k_1}$ of $G_1$ and $B_1, \dots, B_{k_2}$ of $G_2$ such that 
\[
G_1 = \bigcup_{i=1}^{k_1} A_i 
\quad \text{and} \quad 
G_2 = \bigcup_{j=1}^{k_2} B_j.
\]
By adding isolated vertices if necessary, we may assume without loss of generality that each $A_i$ is a spanning subgraph of $G_1$ and each $B_j$ is a spanning subgraph of $G_2$. We construct $k_1+k_2$ word-representable subgraphs of $H$ that cover $H$:

\begin{enumerate}
    \item For each $i \in \{1, \dots, k_1\}$, let $X_i = \mathcal{L}_H(A_i)$. By Theorem~\ref{thm:lexmap-merged}, each $X_i$ is a word-representable spanning subgraph of $H$. Since $G_1 = \bigcup_{i=1}^{k_1} A_i$, it follows from the definition of the lexicographic map that 
    \[
    \bigcup_{i=1}^{k_1} X_i = \big( V(H), \{ (u,x)(v,y) \in E(H) \mid uv \in E(G_1) \} \big).
    \]
    
    \item For each $j \in \{1, \dots, k_2\}$, let $Y_j$ be the spanning subgraph of $H$ with edge set 
    \[
    E(Y_j) = \bigcup_{u \in V(G_1)} \{ (u,x)(u,y) \mid xy \in E(B_j) \}.
    \]
    Because the vertex set of $Y_j$ is $V(H)$ and the supervertices partition $V(H)$, the graph $Y_j$ is precisely a disjoint union of $|V(G_1)|$ copies of the word-representable graph $B_j$. By Theorem~\ref{unionofwrgiswrg}, each $Y_j$ is word-representable.
\end{enumerate}

By the definition of the lexicographic product, any edge $(u,x)(v,y) \in E(H)$ either satisfies $uv \in E(G_1)$ or satisfies $u=v$ with $xy \in E(G_2)$. Since $G_2 = \bigcup_{j=1}^{k_2} B_j$, it follows that 
\[
H = \left(\bigcup_{i=1}^{k_1} X_i\right) \cup \left(\bigcup_{j=1}^{k_2} Y_j\right).
\]
Because each constituent subgraph is word-representable, $H$ admits a covering by $k_1+k_2$ word-representable graphs, yielding $\mu(H) \le k_1+k_2$.

To establish the lower bound, observe that $H = G_1 \circ G_2$ contains induced subgraphs isomorphic to both $G_1$ and $G_2$. Specifically, choosing any fixed vertex $y \in V(G_2)$, the vertex subset $U_1 = \{ (u, y) \mid u \in V(G_1) \}$ contains exactly one vertex from each supervertex, and the induced subgraph $H[U_1]$ is isomorphic to $G_1$. Similarly, choosing any fixed vertex $u \in V(G_1)$, the supervertex $V_u = \{ (u, v) \mid v \in V(G_2) \}$ induces a subgraph $H[V_u]$ isomorphic to $G_2$. By Remark~\ref{monotonic-mu-hereditary-remark}, it follows immediately that $\mu(H) \ge \max\{\mu(G_1), \mu(G_2)\} = \max\{k_1, k_2\}$.

Now suppose $\mathrm{cov}_{\mathrm{comp}}(G_2) \le \max\{k_1,k_2\}$. We show that $\mu(H) \le \max\{k_1, k_2\}$ by analyzing two cases.

\medskip
\noindent
\textbf{Case 1:} $k_1 \ge k_2$. \\
In this case, $\mathrm{cov}_{\mathrm{comp}}(G_2) \le k_1$, meaning there exist comparability subgraphs $C_1, \dots, C_{k_1}$ of $G_2$ such that $G_2 = \bigcup_{i=1}^{k_1} C_i$. By adding isolated vertices if necessary, we may assume without loss of generality that each $C_i$ is a spanning subgraph of $G_2$.
For each $i \in \{1, \dots, k_1\}$, define $W_i = A_i \circ C_i$. If $A_i$ contains at least one edge, Theorem~\ref{lex-prod-char-word-re} ensures each $W_i$ is a word-representable subgraph of $H$. If $A_i$ is an edgeless graph, then by the definition of the lexicographic product, $W_i$ is precisely a disjoint union of $|V(G_1)|$ copies of the comparability graph $C_i$, which is word-representable by Theorem~\ref{unionofwrgiswrg}. Note that each $W_i$ is a spanning subgraph of $H$ since $V(W_i) = V(A_i) \times V(C_i) = V(G_1) \times V(G_2) = V(H)$.

The subgraphs $W_1, \dots, W_{k_1}$ cover all edges $(u,x)(v,y) \in E(H)$ where $uv \in E(G_1)$, as well as all edges satisfying $u=v$ with $xy \in \bigcup_{i=1}^{k_1} E(C_i)$. It follows that $H = \bigcup_{i=1}^{k_1} W_i$, implying $\mu(H) \le k_1 = \max\{k_1, k_2\}$.

\medskip
\noindent
\textbf{Case 2:} $k_2 > k_1$. \\
In this case, $\mathrm{cov}_{\mathrm{comp}}(G_2) \le k_2$, meaning there exist comparability subgraphs $C_1, \dots, C_{k_2}$ of $G_2$ such that $G_2 = \bigcup_{i=1}^{k_2} C_i$. By adding isolated vertices if necessary, we may assume without loss of generality that each $C_i$ is a spanning subgraph of $G_2$. We construct $k_2$ word-representable subgraphs $Z_1, \dots, Z_{k_2}$ of $H$ as follows:
\begin{enumerate}

    \item For each $i \in \{1, \dots, k_1\}$, let $Z_i = A_i \circ C_i$. If $A_i$ contains at least one edge, Theorem~\ref{lex-prod-char-word-re} ensures each $Z_i$ is a word-representable subgraph of $H$. If $A_i$ is an edgeless graph, $Z_i$ is a disjoint union of $|V(G_1)|$ copies of the comparability graph $C_i$, which is word-representable by Theorem~\ref{unionofwrgiswrg}. Note that each $Z_i$ is a spanning subgraph of $H$ since $V(Z_i) = V(A_i) \times V(C_i) = V(G_1) \times V(G_2) = V(H)$.
   
    \item For each $i \in \{k_1+1, \dots, k_2\}$, let $Z_i$ be the spanning subgraph of $H$ with edge set 
    \[
    E(Z_i) = \bigcup_{u \in V(G_1)} \{ (u,x)(u,y) \mid xy \in E(C_i) \}.
    \]
    Since the supervertices partition $V(H)$ and $C_i$ is a spanning subgraph of $G_2$, $Z_i$ is precisely a disjoint union of $|V(G_1)|$ copies of the comparability graph $C_i$. Because comparability graphs are word-representable, Theorem~\ref{unionofwrgiswrg} ensures that each $Z_i$ is word-representable.
\end{enumerate}
The subgraphs $Z_1, \dots, Z_{k_1}$ cover all edges $(u,x)(v,y) \in E(H)$ where $uv \in E(G_1)$, as well as all edges satisfying $u=v$ with $xy \in \bigcup_{i=1}^{k_1} E(C_i)$. The remaining edges satisfying $u=v$ with $xy \in \bigcup_{i=k_1+1}^{k_2} E(C_i)$ are covered by $Z_{k_1+1}, \dots, Z_{k_2}$. It follows that $H = \bigcup_{i=1}^{k_2} Z_{i}$, yielding $\mu(H) \le k_2 = \max\{k_1, k_2\}$. 
Since the matching lower bound $\mu(H) \ge \max\{k_1, k_2\}$ holds generally, the proof is complete.
\end{proof}

The structural characterization established in Theorem~\ref{gen-thm-multi-repno} leaves an open gap when $\operatorname{cov}_{\mathrm{comp}}(G_2) > \max\{\mu(G_1), \mu(G_2)\}$. To understand the behavior of this remaining regime, we examine the first case in which the regime's condition is met. Consider the case where $G_1$ is a word-representable graph containing at least one edge ($\mu(G_1) = 1$), and $G_2$ is a word-representable graph that is not a comparability graph ($\mu(G_2) = 1$). This implies that $\operatorname{cov}_{\mathrm{comp}}(G_2) \ge 2 > \max\{1,1\}$, perfectly satisfying the condition. By Theorem~\ref{lex-prod-char-word-re}, the resulting product graph $H = G_1 \circ G_2$ is non-word-representable, which implies $\mu(H) \ge 2$. Concurrently, our general upper bound dictates that $\mu(H) \le \mu(G_1) + \mu(G_2) = 1 + 1 = 2$. It follows that $\mu(H) = 2$. This demonstrates that the general additive upper bound of Theorem~\ref{gen-thm-multi-repno} is exactly tight when $G_1$ is word-representable, and $G_2$ is a non-comparability word-representable graph.

For the remaining cases where $\max\{\mu(G_1), \mu(G_2)\} \ge 2$, the exact behavior of the multi-word-representation number remains open. Motivated by this, we pose the following question.

\begin{question}\label{q:lex_product_gap}
Let $G_1$ and $G_2$ be graphs such that $\operatorname{cov}_{\mathrm{comp}}(G_2) > \max\{\mu(G_1), \mu(G_2)\}$. What is the exact value of $\mu(G_1 \circ G_2)$? In particular, for higher parameter values (where $\max\{\mu(G_1), \mu(G_2)\} \ge 2$):
\begin{enumerate}
    \item Can the value collapse to the lower bound, meaning $\mu(G_1 \circ G_2) = \max\{\mu(G_1), \mu(G_2)\}$ in any of these cases?
    \item Are there instances where the incremental value $\mu(G_1 \circ G_2) = \max\{\mu(G_1), \mu(G_2)\} + 1$ is achieved?
    \item Can the general additive upper bound $\mu(G_1 \circ G_2) = \mu(G_1) + \mu(G_2)$ become tight?
\end{enumerate}
\end{question}

A complete resolution to Question~\ref{q:lex_product_gap} would fully determine the multi-word-representation number for lexicographic products when $\operatorname{cov}_{\mathrm{comp}}(G_2) > \max\{\mu(G_1), \mu(G_2)\}$.

Theorem~\ref{gen-thm-multi-repno} establishes the general upper bound $\mu(G_1 \circ G_2) \le \mu(G_1) + \mu(G_2)$. For minimal non-word-representable graphs $G_1$ and $G_2$, we have $\mu(G_1) = \mu(G_2) = 2$ (since any such graph $G$ can be covered by the word-representable proper induced subgraph $G \setminus \{v\}$ and a spanning star graph centered at $v$), yielding a lower bound of $\max\{\mu(G_1), \mu(G_2)\} = 2$ and a general upper bound of $4$. Consequently, the value of $\mu(H)$ is bounded within the interval $[2, 4]$. It remains unresolved whether this class of graphs satisfies the exact evaluation condition of Theorem~\ref{gen-thm-multi-repno} (namely, $\operatorname{cov}_{\mathrm{comp}}(G_2) \le 2$, which forces $\mu(H) = 2$) or falls into the remaining open regime outlined in Question~\ref{q:lex_product_gap}. This classification depends directly on the following open problem.

\begin{question}
\label{q:minimal_non_wr_cover}
Let $G$ be a minimal non-word-representable graph. Is $\operatorname{cov}_{\mathrm{comp}}(G) \le 2$\textup{?}
\end{question}

While Question~\ref{q:minimal_non_wr_cover} remains open, the next result establishes a tighter upper bound for this class of graphs. By proving that $\mu(H) \le 3$, we reduce the upper bound from the general case of 4 to 3, narrowing the potential value of $\mu(H)$ to the set $\{2, 3\}$. However, the potential for $\mu(H)=2$ in instances where the condition in Question~\ref{q:minimal_non_wr_cover} is not satisfied remains an open and interesting direction for future study.

\begin{theorem}
Let $G_1$ and $G_2$ be minimal non-word-representable graphs, and let $H = G_1 \circ G_2$. Then $\mu(H) \le 3$.
\end{theorem}

\begin{proof}
Let $G_{1}$ be a graph on $n$ vertices with $V(G_{1}) = \{v_{1}, \ldots, v_{n}\}$, and let $G_{2}$ be a graph on $m$ vertices with $V(G_{2}) = \{u_{1}, \ldots, u_{m}\}$. By the definition of the lexicographic product, $H = G_1 \circ G_2$ is constructed by replacing each vertex $v_i \in V(G_1)$ with a supervertex $V_i$, where the induced subgraph $H[V_i]$ is isomorphic to $G_2$. We classify the edges of $H$ into two categories: \emph{internal edges} (edges within $H[V_i]$) and \emph{cross-edges} (edges connecting distinct supervertices $V_i$ and $V_j$ whenever $v_i v_j \in E(G_1)$).

We demonstrate that $\mu(H) \le 3$ by constructing three spanning subgraphs $H_1, H_2,$ and $H_3$ such that their edge sets satisfy $E(H) = \bigcup_{k=1}^3 E(H_k)$. We then establish that each $H_k$ is word-representable, which implies $\mu(H) \le 3$.

\medskip
\noindent
\textbf{Subgraph Construction:}
Let $v_r \in V(G_1)$ be a distinguished vertex, and let $V_r$ denote its corresponding supervertex in $H$. Fix an arbitrary vertex $w \in V(G_2)$. For every $v_i \in V(G_1)$ with $v_i \neq v_r$, let $r_i \in V_i$ denote the vertex $(v_i, w)$. 

Additionally, select an arbitrary vertex $u_k \in V(G_2)$, and let $u^* = (v_r, u_k)$ denote its exact copy within the supervertex $V_r$.
\begin{enumerate}
    \item \textbf{Subgraph $H_1$:} Define $E(H_1)$ to contain:
    \begin{itemize}
        \item All internal edges of the proper induced subgraph $H[V_r \setminus \{u^*\}]$;
        \item For each $v_i \in V(G_1) \setminus \{v_r\}$, the internal edges of $H[V_i]$ incident to the vertex $r_i$. This forms a star subgraph centered at $r_i$ (along with any isolated vertices not adjacent to $r_i$) in $H_1[V_i]$;
        \item All cross-edges of $H$ that are not incident to any vertex in $V_r$.
    \end{itemize}
    
    \item \textbf{Subgraph $H_2$:} Define $E(H_2)$ to contain all remaining internal edges of $H$ not allocated to $H_1$. Specifically, $E(H_2)$ consists of:
    \begin{itemize}
        \item The internal edges of $H[V_r]$ incident to $u^*$, which form a star subgraph centered at $u^*$ (along with any isolated vertices);
        \item For each $v_i \in V(G_1) \setminus \{v_r\}$, all internal edges of the proper induced subgraph $H[V_i \setminus \{r_i\}]$.
    \end{itemize}
    
    \item \textbf{Subgraph $H_3$:} Define $E(H_3)$ to contain all remaining cross-edges of $H$. These are precisely the cross-edges incident to the distinguished supervertex $V_r$.
\end{enumerate}

By construction, $E(H) = \bigcup_{k=1}^3 E(H_k)$. We now establish the word-representability of each subgraph.

\begin{lemma}
The subgraph $H_1$ is word-representable.
\end{lemma}
\begin{proof}
By construction, no cross-edges incident to $V_r$ are present in $H_1$. Thus, $H_1$ consists of two disconnected components:
\begin{itemize}
    \item \textbf{Component 1:} The subgraph $H_1[V_r]$. By construction, $u^*$ has no incident edges in $H_1$, making it an isolated vertex. Therefore, $H_1[V_r]$ is the disjoint union of $H_1[V_r \setminus \{u^*\}]$ and the isolated vertex $u^*$. Since $H_1[V_r \setminus \{u^*\}]$ is isomorphic to the proper induced subgraph $G_2[V(G_2) \setminus \{u_k\}]$ of the minimal non-word-representable graph $G_2$, it is word-representable. Because the disjoint union of word-representable graphs remains word-representable (Theorem \ref{unionofwrgiswrg}), $H_1[V_r]$ is word-representable.
    \item \textbf{Component 2:} The subgraph induced by $\bigcup_{v_i \neq v_r} V_i$. Within each subgraph $H_1[V_i]$, the internal edges constitute a star graph centered at $r_i$ (along with isolated vertices). Because each $r_i$ corresponds to the same vertex $w \in V(G_2)$, these subgraphs are identical. Thus, this component is isomorphic to the special subgraph $G_{1}' \circ S$, where $G_{1}' = G_1[V(G_1) \setminus \{v_r\}]$ and $S$ is the spanning star subgraph of $G_2$ centered at $w$. Since $G_1$ is minimal non-word-representable, its proper induced subgraph $G_1'$ is word-representable. Furthermore, because $S$ consists exclusively of edges incident to the single vertex $w$, it contains no cycles. As a forest, it is a comparability graph. Thus, by Definition \ref{defn-special}, this product graph is a special subgraph of $G_{1} \circ G_{2}$, and by Observation \ref{special-word-rep-comp-char}, it is word-representable.
\end{itemize}
By Theorem \ref{word-rep-components-thm}, $H_1$ is word-representable.
\end{proof}

\begin{lemma}
The subgraph $H_2$ is word-representable.
\end{lemma}
\begin{proof}
Since $H_2$ contains no cross-edges, it is the disjoint union of the induced subgraphs $H_2[V_i]$ for all $v_i \in V(G_1)$.
\begin{itemize}
    \item For $v_i = v_r$, $H_2[V_r]$ is a forest consisting of a star graph centered at $u^*$ and a set of isolated vertices. Every forest is a comparability graph and is therefore word-representable.
    \item For each $v_i \in V(G_1) \setminus \{v_r\}$, $r_i$ acts as an isolated vertex in $H_2[V_i]$. The remaining vertices induce a subgraph isomorphic to $G_2[V(G_2) \setminus \{w\}]$. Since $G_2$ is minimal non-word-representable, this proper induced subgraph is word-representable, and the inclusion of the isolated vertex $r_i$ preserves this property via disjoint union.
\end{itemize}
As the disjoint union of word-representable graphs is word-representable, $H_2$ is word-representable (Theorem \ref{unionofwrgiswrg}).
\end{proof}

\begin{lemma}
The subgraph $H_3$ is word-representable.
\end{lemma}
\begin{proof}
The edge set $E(H_3)$ consists precisely of all cross-edges incident to the distinguished supervertex $V_r$. The graph $H_3$ is the lexicographic map of a word-representable subgraph of $G_1$ (specifically, the spanning star subgraph $S_{G_1}$ of $G_1$, where $E(S_{G_1})$ consists exactly of all edges in $G_1$ incident to $v_r$, including any vertices not adjacent to $v_r$ as isolated vertices). Consequently, by Theorem \ref{thm:lexmap-merged}, $H_3$ is word-representable.
\end{proof}

Since $E(H) = \bigcup_{k=1}^3 E(H_k)$ and each subgraph $H_k$ is word-representable, we conclude that $\mu(H) \le 3$.
\end{proof}

Having established that $\mu(G_1 \circ G_2) \le 3$, it is natural to ask whether this bound is tight. This leads to the following open problem regarding the existence of minimal non-word-representable graphs whose lexicographic product attains this upper bound.

\begin{question}
\label{q:tightness_mu_3}
Do there exist minimal non-word-representable graphs $G_1$ and $G_2$ such that $\mu(G_1 \circ G_2) = 3$\textup{?}
\end{question}

\section{Lexicographic powers of graph and multi-word-representation number}
\label{sec:lexpower_and_mul}

In this section, we provide bounds on the multi-word-representation number for the lexicographic powers of a graph. By leveraging the known characterization of word-representability for lexicographic products, we first deduce the following characterization of word-representability for lexicographic powers.

\begin{theorem}
\label{lex-power-charact}
Let $G$ be a graph and let $k \ge 2$. Then:
\begin{enumerate}
    \item If $G$ is a comparability graph, then $G^{[k]}$ is a comparability graph and hence word-representable.
    \item If $G$ is not a comparability graph, then $G^{[k]}$ is not word-representable.
\end{enumerate}
\end{theorem}

\begin{proof}
Suppose first that $G$ is a comparability graph. We show by induction on $k \ge 2$ that $G^{[k]}$ is a comparability graph, and hence word-representable. 
For the base case $k=2$, we have $G^{[2]} = G \circ G$, which is a comparability graph by Theorem~\ref{lex-prod-char-compara}. 
For the inductive step, assume that $G^{[r]}$ is a comparability graph for some $r \ge 2$. By definition, $G^{[r+1]} = G^{[r]} \circ G$. Since both $G^{[r]}$ and $G$ are comparability graphs, it follows again from Theorem~\ref{lex-prod-char-compara} that $G^{[r+1]}$ is a comparability graph. This completes the induction.

Now suppose that $G$ is not a comparability graph. For $k=2$, $G^{[2]} = G \circ G$ is not word-representable by Theorem~\ref{lex-prod-char-word-re}. For $k > 2$, writing $G^{[k]} = G^{[k-1]} \circ G$ by definition, and noting that $G$ is not a comparability graph, Theorem~\ref{lex-prod-char-word-re} implies that $G^{[k]}$ is not word-representable for every $k \ge 2$.
\end{proof}

From Theorem~\ref{lex-power-charact}, if $G$ is a comparability graph, then $\mu(G^{[k]}) = 1$ for all $k \geq 1$. In contrast, when $G$ is not a comparability graph, the theorem provides no information about the multi-word-representation number of $G^{[k]}$ for $k \geq 2$.

To address this case, we establish a general upper bound in terms of a structural parameter of $G$. 

\begin{theorem}
\label{lex-power-mu-gen-bound}
For any graph $G$ and any integer $k \ge 2$,
\[
\mu(G^{[k]}) \le \mathrm{cov}_{\mathrm{comp}}(G).
\]
\end{theorem}

\begin{proof}
Let $c = \mathrm{cov}_{\mathrm{comp}}(G)$, and let $\{A_1, A_2, \dots, A_c\}$ be a collection of comparability subgraphs of $G$ such that $G = \bigcup_{i=1}^c A_i$. By adding isolated vertices if necessary, we may assume without loss of generality that each $A_i$ is a spanning subgraph. We proceed by induction on $k \ge 2$ to show that $G^{[k]}$ can be covered by $c$ spanning comparability subgraphs.

\medskip
\noindent
\textbf{Base case: $k=2$.}  
Consider $G^{[2]} = G \circ G$. For each $i \in \{1, \dots, c\}$, let $C_i = A_i \circ A_i$. Since $A_i$ is a spanning subgraph of $G$, $C_i$ is a spanning subgraph of $G^{[2]}$. Because $A_i$ is a comparability graph, Theorem~\ref{lex-prod-char-compara} ensures each $C_i$ is a comparability graph. 

Let $e = (u,x)(v,y)$ be an arbitrary edge in $G^{[2]}$. By the definition of the lexicographic product, two cases arise:
\begin{enumerate}
    \item If $uv \in E(G)$, then $uv \in E(A_i)$ for some $i$. Since $A_i$ is a spanning subgraph, $x, y \in V(A_i)$ for all $x,y \in V(G)$, implying $e \in E(C_i)$.
    \item If $u = v$ and $xy \in E(G)$, then $xy \in E(A_i)$ for some $i$. Since $A_i$ is a spanning subgraph, $u \in V(A_i)$, implying $e \in E(C_i)$.
\end{enumerate}
Therefore, every edge of $G^{[2]}$ belongs to some $C_i$, so $G^{[2]} = \bigcup_{i=1}^c C_i$.

\medskip
\noindent
\textbf{Inductive step.} 
Assume that for some $r \ge 2$, $G^{[r]} = \bigcup_{i=1}^c B_i$, where each $B_i$ is a spanning comparability subgraph of $G^{[r]}$. Consider $G^{[r+1]} = G \circ G^{[r]}$. For each $i \in \{1, \dots, c\}$, let $C_i = A_i \circ B_i$. Since $A_i$ and $B_i$ are spanning subgraphs of $G$ and $G^{[r]}$ respectively, $C_i$ is a spanning subgraph of $G^{[r+1]}$. Because both $A_i$ and $B_i$ are comparability graphs, Theorem~\ref{lex-prod-char-compara} ensures each $C_i$ is a comparability graph.

Let $e = (u,x)(v,y)$ be an arbitrary edge in $G^{[r+1]}$. By the definition of the lexicographic product, two cases arise:
\begin{enumerate}
    \item If $uv \in E(G)$, then $uv \in E(A_i)$ for some $i$. Since $B_i$ is a spanning subgraph of $G^{[r]}$, we have $x, y \in V(B_i)$, which implies $e \in E(C_i)$.
    \item If $u = v$ and $xy \in E(G^{[r]})$, then $xy \in E(B_i)$ for some $i$. Since $A_i$ is a spanning subgraph of $G$, we have $u \in V(A_i)$, which implies $e \in E(C_i)$.
\end{enumerate}
Therefore, every edge of $G^{[r+1]}$ belongs to some $C_i$, so $G^{[r+1]} = \bigcup_{i=1}^c C_i$.

\medskip
By induction, $G^{[k]}$ is a union of at most $c$ spanning comparability subgraphs for all $k \ge 2$. Because every comparability graph is word-representable, this collection also constitutes a valid cover of $G^{[k]}$ by at most $c$ word-representable subgraphs. Therefore, the multi-word-representation number $\mu(G^{[k]})$ is bounded above by $c$, and we conclude that $\mu(G^{[k]}) \le c = \mathrm{cov}_{\mathrm{comp}}(G)$.
\end{proof}

Theorem~\ref{lex-power-mu-gen-bound} provides a general upper bound on the multi-word-representation number of $G^{[k]}$ for any integer $k \ge 2$. A key feature of this bound is its independence from $k$. Consequently, this bound is particularly powerful when $\mathrm{cov}_{\mathrm{comp}}(G)$ is small or bounded by a constant, or when $k$ grows arbitrarily large.

However, when $\mathrm{cov}_{\mathrm{comp}}(G)$ is large, the above bound may not be optimal, especially for small values of $k$. In such cases, a better bound can be obtained under the additional assumption that $G$ is a word-representable non-comparability graph. This is established in the following theorem.

\begin{theorem}
Let $G$ be a word-representable graph that is not a comparability graph, and let $k \ge 2$. Then
\[
\mu(G^{[k]}) \le k.
\]
\end{theorem}

\begin{proof}
We prove that $\mu(G^{[k]}) \le k$ for all $k \ge 2$ by induction on $k$.

\medskip
\noindent
\textbf{Base case: $k=2$.}  
We have $G^{[2]} = G \circ G$. By Theorem~\ref{lex-power-charact}, it follows that $\mu(G^{[2]}) \ge 2$. We now demonstrate that $\mu(G^{[2]}) \le 2$ by covering $G^{[2]}$ with two spanning word-representable subgraphs. Let $A = \mathcal{L}_{G^{[2]}}(G)$, and let $B$ be the spanning subgraph of $G^{[2]}$ with vertex set $V(B) = V(G^{[2]})$ and edge set 
\[
E(B) = \bigcup_{u \in V(G)} \{ (u,x)(u,y) \mid xy \in E(G) \}.
\]
By construction, every edge of $G^{[2]}$ belongs to either $A$ or $B$, so $G^{[2]} = A \cup B$. By Theorem~\ref{thm:lexmap-merged}, the graph $A$ is word-representable. Let $V_u = \{(u,x) \mid x \in V(G)\}$. Since $\{V_u : u \in V(G)\}$ is a partition of $V(G^{[2]})$, the graph $B$ is a vertex-disjoint union of copies of $G$. Hence, $B$ is a disjoint union of word-representable graphs and is therefore word-representable by Theorem~\ref{unionofwrgiswrg}. It follows that $\mu(G^{[2]}) \le 2$, implying $\mu(G^{[2]}) = 2$.

\medskip
\noindent
\textbf{Inductive step.} 
Assume that $\mu(G^{[r]}) \le r$ for some $r \ge 2$, meaning $G^{[r]} = \bigcup_{i=1}^{r} A_i$ for some spanning word-representable subgraphs $A_1, \dots, A_r$ of $G^{[r]}$. Consider $G^{[r+1]} = G \circ G^{[r]}$. We construct $r+1$ spanning subgraphs $H_1, \dots, H_{r+1}$ of $G^{[r+1]}$ with vertex set $V(G^{[r+1]})$ as follows:

\begin{enumerate}
    \item For each $i \in \{1, \dots, r\}$, let $H_i$ be the spanning subgraph with edge set 
    \[
    E(H_i) = \bigcup_{u \in V(G)} \{ (u,x)(u,y) \mid xy \in E(A_i) \}.
    \]
    Let $V_u = \{(u,x) \mid x \in V(G^{[r]})\}$. Since $\{V_u : u \in V(G)\}$ is a partition of $V(G^{[r+1]})$, each $H_i$ is a vertex-disjoint union of copies of $A_i$. Hence, each $H_i$ is word-representable by Theorem~\ref{unionofwrgiswrg}.
    
    \item Let $H_{r+1} = \mathcal{L}_{G^{[r+1]}}(G)$. By Theorem~\ref{thm:lexmap-merged}, since $G$ is word-representable, $H_{r+1}$ is word-representable.
\end{enumerate}

By construction, any edge $(u,x)(v,y) \in E(G^{[r+1]})$ satisfying $u=v$ requires that $xy \in E(G^{[r]})$, which is covered by the subgraph union $\bigcup_{i=1}^{r} H_i$. Alternatively, if $u \neq v$, then $uv \in E(G)$, and the edge is covered by $H_{r+1}$. Thus, $G^{[r+1]} = \bigcup_{j=1}^{r+1} H_j$. Since each $H_j$ is a spanning word-representable subgraph of $G^{[r+1]}$, we conclude that $\mu(G^{[r+1]}) \le r+1$, completing the inductive step.
\end{proof}

\section{Asymptotic bounds on $\tau(n)$ and the WB problem}
\label{sec:tau:n}

In this section, we investigate the asymptotic behavior of the extremal function $\tau(n)$ and apply our findings to the Word-representable Bipartition (WB) problem. We first construct a family of graphs using the lexicographic powers of a carefully chosen base graph to establish a sublinear general upper bound for $\tau(n)$. We then highlight a sharp structural contrast between $\tau(n)$ and the multi-word-representation number $\mu(G)$ for this family. Utilizing this extremal family, we prove that the WB problem has a negative answer for all sufficiently large $n$. Finally, we contrast our sublinear upper bound by providing lower bounds for $\tau(n)$ in general graphs and a polynomial lower bound restricted to perfect graphs.

It was shown in~\cite{multi-word-rep} that there exist three graphs on $8$ vertices that do not contain any word-representable induced subgraph on $7$ vertices. Let $H$ be one such graph, as depicted in Figure~\ref{fig:tau8neq7}. By this selection, the base graph satisfies $|V(H)| = 8$ and $\eta(H) \le 6$. 

We first determine an upper bound on the maximum order of a word-representable induced subgraph in the $k$-th lexicographic power of $H$.

\begin{lemma}
\label{etaHk}
For every integer $k \ge 1$, we have $\eta(H^{[k]}) \le 6^{k}$.
\end{lemma}

\begin{proof}
We proceed by induction on $k$. For $k=1$, we have $H^{[1]} = H$. By our choice of $H$, it contains no word-representable induced subgraph on $7$ vertices, which directly implies $\eta(H) \le 6$.

Assume that $\eta(H^{[r]}) \le 6^{r}$ for some $r \ge 1$, and consider $H^{[r+1]} = H \circ H^{[r]}$. By the structure of the lexicographic product, $H^{[r+1]}$ consists of $8$ supervertices $V_1, \dots, V_8$, where each induced subgraph $H^{[r+1]}[V_i]$ is isomorphic to $H^{[r]}$, and adjacency between distinct supervertices is determined by the edges of $H$.

Let $S$ be a representable set in $H^{[r+1]}$. For each $i$, the set $S \cap V_i$ is representable in $H^{[r+1]}[V_i]$, so $|S \cap V_i| \le 6^{r}$ by the induction hypothesis. Since the supervertices $V_1, \dots, V_8$ form a partition of $V(H^{[r+1]})$, we have
\[
|S| = \sum_{i=1}^{8} |S \cap V_i|.
\]

We claim that $S$ intersects at most $6$ supervertices. Suppose, for contradiction, that $S$ intersects at least $7$ distinct supervertices. Let $I \subseteq \{1,\dots,8\}$ be a set of seven indices such that $S \cap V_i \neq \varnothing$ for every $i \in I$. For each $i \in I$, choose a vertex $u_i \in S \cap V_i$, and let $U=\{u_i : i\in I\}$.

By the definition of the lexicographic product, for distinct $i,j\in I$, the vertices $u_i$ and $u_j$ are adjacent in $H^{[r+1]}$ if and only if the corresponding vertices $i$ and $j$ are adjacent in $H$. Therefore,
\[
H^{[r+1]}[U]\cong H[I].
\]

Since $S$ is representable, the induced subgraph $H^{[r+1]}[S]$ is word-representable. As the class of word-representable graphs is hereditary and $U\subseteq S$, it follows that $H^{[r+1]}[U]$ is word-representable. Consequently, $H[I]$ is word-representable.

However, by the choice of $H$, no induced subgraph of $H$ on $7$ vertices is word-representable. Since $|I|=7$, the graph $H[I]$ is not word-representable, a contradiction. Therefore, $S$ intersects at most $6$ supervertices. Hence,
\[
|S| \le 6 \cdot 6^{r} = 6^{r+1}.
\]
Thus, $\eta(H^{[r+1]}) \le 6^{r+1}$, completing the induction.
\end{proof}

\begin{theorem}
\label{tau8k-6k}
For any positive integer $k \ge 1$, $\tau(8^{k}) \le 6^{k}$.
\end{theorem}

\begin{proof}
Let $G = H^{[k]}$ be the $k$-th lexicographic power of the base graph $H$ defined above. Since $|V(H)| = 8$, the graph $G$ has exactly $8^k$ vertices.

From the definition of $\tau(n)$, we have $\tau(n) \le \eta(G)$ for every graph $G$ on $n$ vertices. Applying this to $G = H^{[k]}$ and utilizing Lemma~\ref{etaHk}, we obtain
\[
\tau(8^{k}) \le \eta(H^{[k]}) \le 6^{k}.
\]
\end{proof}

\begin{figure}[h]
\centering
\begin{tikzpicture}[scale=0.75, transform shape]
    \tikzset{hollow/.style = {shape=circle,draw=black,fill=white,minimum size=1.5em,text=black}}
    \tikzset{edge/.style = {draw=black, - = latex'}}
    
    \node[hollow] (1) at  (0, 0) {$1$};
    \node[hollow] (2) at  (90:2cm) {$2$};
    \node[hollow] (3) at  (162:2cm) {$3$};
    \node[hollow] (4) at  (234:2cm) {$4$};
    \node[hollow] (5) at  (306:2cm) {$5$};
    \node[hollow] (6) at  (378:2cm) {$6$};
    \node[hollow] (7) at  (-2.6,-3.6) {$7$};
    \node[hollow] (8) at  (2.6, -3.6) {$8$};
    
    \draw[edge] (1) -- (2); 
    \draw[edge] (1) -- (3); 
    \draw[edge] (1) -- (4); 
    \draw[edge] (1) -- (5); 
    \draw[edge] (1) -- (6);

    \draw[edge] (2) -- (3); 
    \draw[edge] (2) -- (6); 
    \draw[edge] (2) -- (7); 

    \draw[edge] (3) -- (4); 
    \draw[edge] (3) -- (7); 
    \draw[edge] (3) -- (8); 

    \draw[edge] (4) -- (5); 
    \draw[edge] (4) -- (8);

    \draw[edge] (5) -- (6);
    \draw[edge] (5) -- (8);

    \draw[edge] (6) -- (7);
    \draw[edge] (6) -- (8);

    \draw[edge] (7) -- (8);
\end{tikzpicture}
\caption{An $8$-vertex graph $H$ with no $7$-vertex word-representable induced subgraph.}
\label{fig:tau8neq7}
\end{figure}

\begin{corollary}
\label{taun-bound-corollary}
For any $\epsilon > 0$ and sufficiently large $n$, $\tau(n) \le n^{\log_8 6 + \epsilon}$.
\end{corollary}

\begin{proof}
Let $k$ be the unique positive integer such that $8^{k-1} < n \le 8^k$.
Then $k - 1 < \log_8 n \le k$, so $k < \log_8 n + 1$.

Since $\tau$ is non-decreasing \cite{multi-word-rep}, we have
$\tau(n) \le \tau(8^k) \le 6^k$.
Hence,
\[
\tau(n) \le 6^k < 6^{\log_8 n + 1} = 6n^{\log_8 6}.
\]

For any $\epsilon > 0$, for sufficiently large $n$ we have $6 \le n^\epsilon$, and thus
\[
\tau(n) \le 6n^{\log_8 6} \le n^{\log_8 6 + \epsilon}.
\]
\end{proof}
Since $\log_8 6 < 1$, Corollary~\ref{taun-bound-corollary} implies that $\tau(n)$ grows sublinearly.
As an interesting application of Theorem~\ref{lex-power-mu-gen-bound}, we obtain the following precise evaluation.

\begin{observation}
\label{obs-mu-H-k}
Let $H$ be the graph depicted in Figure~\ref{fig:tau8neq7}. Then, for every integer $k \ge 1$,
\[
\mu(H^{[k]}) = 2.
\]
\end{observation}

\begin{proof}
We first determine $\mathrm{cov}_{\mathrm{comp}}(H)$. Since $H$ is non-word-representable~\cite{multi-word-rep}, $H$ is a non-comparability graph. Hence, $\mathrm{cov}_{\mathrm{comp}}(H) \ge 2$.

Figure~\ref{fig:4coloringH} demonstrates a valid $4$-coloring of the graph $H$, where the four independent sets formed in the partition are $\{1,8\}$, $\{2, 4\}$, $\{3, 6\}$, and $\{5, 7\}$. The chromatic number of $H$ is thus at most $4$. Furthermore, every $3$-colorable graph is word-representable~\cite{Halldorsson2011}. Since $H$ is known to be non-word-representable~\cite{multi-word-rep}, it cannot be $3$-colorable, implying that $\chi(H) \ge 4$. Consequently, we have exactly $\chi(H) = 4$. 

Every bipartite graph is a comparability graph~\cite{golumbic2004algorithmic}; thus, any covering of $H$ by bipartite graphs is inherently a covering by comparability graphs, yielding the inequality $\mathrm{cov}_{\mathrm{comp}}(H) \le \mathrm{cov}_{\mathrm{bip}}(H)$, where $\mathrm{cov}_{\mathrm{bip}}(H)$ denotes the cover number of $H$ by bipartite graphs. Further, $\mathrm{cov}_{\mathrm{bip}}(H) = \lceil \log_2 \chi(H) \rceil$~\cite{harary1977biparticity}. Substituting $\chi(H) = 4$ yields $\mathrm{cov}_{\mathrm{bip}}(H) = \lceil \log_2 4 \rceil = 2$, which implies $\mathrm{cov}_{\mathrm{comp}}(H) \le 2$. Combining our evaluations yields the exact value $\mathrm{cov}_{\mathrm{comp}}(H) = 2$.

Now, we determine the multi-word-representation number $\mu(H^{[k]})$. Since $H$ is an induced subgraph of its lexicographic power $H^{[k]}$, Remark~\ref{monotonic-mu-hereditary-remark} implies that $\mu(H^{[k]}) \ge \mu(H)$. Because $H$ is non-word-representable, $\mu(H) \ge 2$. We therefore obtain $\mu(H^{[k]}) \ge 2$ for all $k \ge 1$. 

For the upper bound, we branch our analysis into two cases based on the exponent $k$. If $k = 1$, then $H^{[1]} = H$; since every comparability graph is word-representable, any covering by comparability graphs constitutes a valid covering by word-representable graphs, yielding $\mu(H) \le \mathrm{cov}_{\mathrm{comp}}(H) = 2$. If $k \ge 2$, Theorem~\ref{lex-power-mu-gen-bound} directly provides the upper bound $\mu(H^{[k]}) \le \mathrm{cov}_{\mathrm{comp}}(H) = 2$. In both instances, the upper bound matches the lower bound exactly, completing the proof that $\mu(H^{[k]}) = 2$ for all $k \ge 1$.
\end{proof}

\begin{remark}
\label{rem:parameter-comparison}
Lemma~\ref{etaHk} and Observation~\ref{obs-mu-H-k} indicate that the parameters $\mu(G)$ and $\eta(G)$ capture different aspects of the structure of a graph. In particular, a graph family may have a bounded multi-word-representation number while the order of its largest word-representable induced subgraph is asymptotically negligible relative to the order of the whole graph. Indeed, for every integer $k \ge 1$, we have $\mu(H^{[k]})=2$ and $\eta(H^{[k]}) \le 6^k$, whereas $|V(H^{[k]})|=8^k$. 
\end{remark}

\begin{figure}[h]
\centering
\begin{tikzpicture}[scale=0.75, transform shape]

\tikzset{
    hollow/.style = {circle, draw=black, minimum size=1.5em},
    edge/.style = {draw=black}
}


\node[hollow, fill=red!60] (1) at  (0, 0) {$1$};
\node[hollow, fill=blue!60] (2) at  (90:2cm) {$2$};
\node[hollow, fill=green!60] (3) at  (162:2cm) {$3$};
\node[hollow, fill=blue!60] (4) at  (234:2cm) {$4$};
\node[hollow, fill=yellow!70] (5) at  (306:2cm) {$5$};
\node[hollow, fill=green!60] (6) at  (378:2cm) {$6$};
\node[hollow, fill=yellow!70] (7) at  (-2.6,-3.6) {$7$};
\node[hollow, fill=red!60] (8) at  (2.6,-3.6) {$8$};


\draw[edge] (1) -- (2); 
\draw[edge] (1) -- (3); 
\draw[edge] (1) -- (4); 
\draw[edge] (1) -- (5); 
\draw[edge] (1) -- (6);

\draw[edge] (2) -- (3); 
\draw[edge] (2) -- (6); 
\draw[edge] (2) -- (7); 

\draw[edge] (3) -- (4); 
\draw[edge] (3) -- (7); 
\draw[edge] (3) -- (8); 

\draw[edge] (4) -- (5); 
\draw[edge] (4) -- (8);

\draw[edge] (5) -- (6);
\draw[edge] (5) -- (8);

\draw[edge] (6) -- (7);
\draw[edge] (6) -- (8);

\draw[edge] (7) -- (8);

\end{tikzpicture}
\caption{A proper 4-coloring of the graph $H$.}
\label{fig:4coloringH}
\end{figure}

We now apply the extremal properties of the family $H^{[k]}$ to resolve the Word-representable Bipartition (WB) problem introduced in~\cite{multi-word-rep}. Observation 6 in \cite{multi-word-rep} states that if the Word-representable Bipartition (WB) problem is solved affirmatively for $n$, then any graph $G$ on at most $n + \tau(n)$ vertices is $2$-multi-word-representable. The proof technique utilized therein offers a strategy for establishing $2$-multi-word-representability for such a graph $G$: isolate a word-representable induced subgraph $A$ of order $\tau(n)$ and verify whether the remaining induced subgraph $G \setminus A$ admits a word-representable bipartition (that is, its vertex set can be partitioned into two subsets, each inducing a word-representable subgraph). 

While this strategy provides a sufficient condition for establishing $2$-multi-word-representability, it is not universally applicable. In fact, there exist infinite families of graphs for which this technique inherently fails. Specifically, for the family $H^{[k]}$ where $k \ge 4$, although $H^{[k]}$ is $2$-multi-word-representable (Observation~\ref{obs-mu-H-k}), the following corollary demonstrates that for any word-representable induced subgraph $A$ (such as one of order $\tau(n)$), the induced subgraph $H^{[k]} \setminus A$ cannot be vertex-partitioned into two word-representable induced subgraphs.

\begin{corollary}
\label{cor:wb_limitation}
For any integer $k \ge 4$, let $G = H^{[k]}$ and let $A$ be any word-representable induced subgraph of $G$. The induced subgraph $G \setminus A$ cannot be vertex-partitioned into two word-representable induced subgraphs.
\end{corollary}

\begin{proof}
Let $k \ge 4$, and let $S \subseteq V(G)$ be any set of vertices such that the induced subgraph $G[S]$ is word-representable. By Lemma~\ref{etaHk}, the order of any word-representable induced subgraph in $G$ is at most $6^k$. Therefore, $|S| \le 6^k$.

Consider the induced subgraph $G' = G \setminus S$. The order of $G'$ satisfies:
\[
|V(G')| = |V(G)| - |S| \ge 8^k - 6^k.
\]

Assume, for the sake of contradiction, that $G'$ admits a word-representable bipartition. That is, $V(G')$ can be partitioned into two subsets, $V_1$ and $V_2$, such that both $G'[V_1]$ and $G'[V_2]$ are word-representable. By the Pigeonhole Principle, at least one of these subsets, say $V_1$, satisfies:
\[
|V_1| \ge \frac{|V(G')|}{2} \ge \frac{8^k - 6^k}{2}.
\]

Since $V_1 \subseteq V(G)$ and the induced subgraph $G[V_1]$ is word-representable, its order is bounded by the maximum order of any word-representable induced subgraph in $G$, which is $6^k$ by Lemma~\ref{etaHk}. Combining the lower bound for $|V_1|$ with this upper bound yields:
\[
\frac{8^k - 6^k}{2} \le 6^k \implies 8^k - 6^k \le 2 \cdot 6^k \implies 8^k \le 3 \cdot 6^k.
\]
Dividing by $6^k$, we obtain:
\[
\left(\frac{4}{3}\right)^k \le 3.
\]
For $k = 4$, we have $(4/3)^4 = 256/81 \approx 3.16 > 3$. Since $f(k) = (4/3)^k$ is a strictly increasing function, the inequality $(4/3)^k \le 3$ is false for all $k \ge 4$. This contradiction completes the proof.
\end{proof}

\begin{theorem} \label{thm:wb_negative}
    For every integer $n \ge 433$, there exists a graph of order $n$ that does not admit a word-representable bipartition.
\end{theorem}

\begin{proof}
    Let $G = H^{[3]}$, where $H$ is the graph depicted in Figure~\ref{fig:tau8neq7}. By Theorem~\ref{tau8k-6k} and Lemma~\ref{etaHk}, $G$ has order $|V(G)| = 8^3 = 512$, and the maximum order of any word-representable induced subgraph of $G$ is at most $6^3 = 216$.

    We first show that no induced subgraph of $G$ of order at least $433$ admits a word-representable bipartition. Let $G'$ be an arbitrary induced subgraph of $G$ of order $|V(G')| \ge 2 \cdot 6^3 + 1 = 433$. Assume, for the sake of contradiction, that $G'$ admits a word-representable bipartition. That is, $V(G')$ can be partitioned into two subsets, $V_1$ and $V_2$, such that the induced subgraphs $G'[V_1]$ and $G'[V_2]$ are both word-representable. 

    By the Pigeonhole Principle, at least one of these subsets, say $V_1$, must contain at least half of the vertices of $G'$. Thus,
    \[
        |V_1| \ge \left\lceil \frac{|V(G')|}{2} \right\rceil \ge \left\lceil \frac{433}{2} \right\rceil = 217.
    \]
    Because $V_1 \subseteq V(G')$ and $G'$ is an induced subgraph of $G$, we have $G'[V_1] = G[V_1]$. As $G[V_1]$ is a word-representable induced subgraph, its order is at most $\eta(G)$. By Lemma~\ref{etaHk}, this implies $|V_1| \le 216$, yielding the contradiction $217 \le |V_1| \le 216$. Therefore, $G'$ cannot admit a word-representable bipartition.

    Because $G$ is a graph on $512$ vertices, taking an arbitrary induced subgraph of $G$ of order $n \in [433, 512]$ provides a graph on $n$ vertices that lacks a word-representable bipartition.

    To extend this result to all integers $n > 512$, we observe that the property of admitting a word-representable bipartition is hereditary. Indeed, if a graph $W$ admits a vertex partition $(U_1, U_2)$ such that $W[U_1]$ and $W[U_2]$ are word-representable, then for any induced subgraph $X$ of $W$, the restricted partition $(U_1 \cap V(X), U_2 \cap V(X))$ ensures that the subgraphs $X[U_1 \cap V(X)]$ and $X[U_2 \cap V(X)]$ are also word-representable. Consequently, if a graph lacks a word-representable bipartition, then any graph containing it as an induced subgraph must also lack one. Thus, for any $n > 512$, constructing a graph by taking the disjoint union of $G$ and $n - 512$ isolated vertices yields a graph on $n$ vertices that satisfies the theorem.
\end{proof}

Consequently, Theorem~\ref{thm:wb_negative} resolves the Word-representable Bipartition (WB) problem in the negative for all $n \ge 433$.

\subsection{Lower bounds for general and perfect graphs}

Having established a sublinear upper bound for $\tau(n)$, we now turn our attention to general lower bounds. A classical Ramsey-theoretic argument implies that every graph on $n$ vertices contains either a clique or an independent set of order at least $c \log n$ for some absolute constant $c > 0$ (see, e.g., \cite{diestel}). Since both cliques and independent sets are word-representable (Example \ref{example-word-rep-graph}), it follows that every graph on $n$ vertices contains a word-representable induced subgraph of order at least $c \log n$. Consequently,
\[
\tau(n) \ge c \log n.
\]

We now examine the behavior of this parameter restricted to the class of perfect graphs. Let $\tau_{\mathrm{perf}}(n)$ denote the minimum value of $\eta(G)$ taken over all perfect graphs of order $n$:
\[
\tau_{\mathrm{perf}}(n) = \min\{\eta(G) : G \text{ is a perfect graph on } n ~\text{vertices}\}.
\]

A graph $G$ is \emph{perfect} if $\chi(H) = \omega(H)$ for every induced subgraph $H \subseteq G$. Let $G$ be a perfect graph of order $n$. Since $\chi(G) = \omega(G)$, any proper $\chi(G)$-coloring of $G$ induces a partition of $V(G)$ into $\chi(G)$ independent sets, at least one of which must have cardinality at least $\lceil n/\chi(G) \rceil$. Consequently, $\alpha(G) \ge n/\chi(G) = n/\omega(G)$, which implies
\[
\alpha(G)\omega(G) \ge n.
\]
Utilizing the inequality $\max\{a, b\} \ge \sqrt{ab}$ for positive reals $a, b$, we obtain
\[
\max\{\alpha(G), \omega(G)\} \ge \sqrt{\alpha(G)\omega(G)} \ge n^{1/2}.
\]
Thus, every perfect graph of order $n$ contains either a clique or an independent set of order at least $n^{1/2}$. Since both cliques and independent sets are word-representable, it follows that
\[
\tau_{\mathrm{perf}}(n) \ge n^{1/2}.
\]

\section{Corona product of graphs and multi-word-representation number}
\label{sec:corona}

In this section, we investigate the multi-word-representation number of the corona product $G_1 \odot G_2$ and establish its relationship with $\mu(G_1)$ and $\mu(G_2)$.

\begin{theorem}
\label{corona-gen-thm}
Let $G_1$ and $G_2$ be graphs, and let $H = G_1 \odot G_2$. Then
\[
\max\{\mu(G_1), \mu(G_2)\} \le \mu(H) \le \max\{\mu(G_1), \mu(G_2)\} + 1.
\]
Furthermore, if either
\begin{enumerate}
    \item $\mu(G_1) > \mu(G_2)$, or
    \item $G_2$ admits a covering by $\mu(G_2)$ word-representable graphs in which at least one graph is a comparability graph,
\end{enumerate}
then $
\mu(H)=\max\{\mu(G_1),\mu(G_2)\}.
$
\end{theorem}

\begin{proof}
By the definition of the corona product $H = G_1 \odot G_2$, the graph $G_1$ is an induced subgraph of $H$. By Remark~\ref{monotonic-mu-hereditary-remark}, it follows that $\mu(G_1) \leq \mu(H)$. Similarly, for any fixed vertex $v \in V(G_1)$, the copy of $G_2$ attached to $v$, denoted by $G_2^v$, forms an induced subgraph of $H$ that is isomorphic to $G_2$. Applying Remark~\ref{monotonic-mu-hereditary-remark} again yields $\mu(G_2) \leq \mu(H)$. Consequently, we obtain the lower bound:
\[
\mu(H) \geq \max\{\mu(G_1), \mu(G_2)\}.
\]

Let $\mu(G_{1}) = k_{1}$ and $\mu(G_{2}) = k_{2}$, and let $k = \max\{k_{1}, k_{2}\}$. By definition, $G_{1}$ can be expressed as the union of $k_{1}$ word-representable subgraphs $A_{1}, A_{2}, \ldots, A_{k_{1}}$, and $G_{2}$ can be expressed as the union of $k_{2}$ word-representable subgraphs $B_{1}, B_{2}, \ldots, B_{k_{2}}$. Without loss of generality, we assume that $V(A_{i}) = V(G_{1})$ for all $1 \leq i \leq k_{1}$ and $V(B_{j}) = V(G_{2})$ for all $1 \leq j \leq k_{2}$. For indices satisfying $k_{1} < i \leq k$ and $k_{2} < j \leq k$ (if any), we define $A_{i} = (V(G_{1}), \emptyset)$ and $B_{j} = (V(G_{2}), \emptyset)$, each of which is an empty spanning subgraph and is trivially word-representable. Thus, we can write $G_{1} = \bigcup_{i=1}^{k} A_{i}$ and $G_{2} = \bigcup_{j=1}^{k} B_{j}$, where each $A_{i}$ and $B_{j}$ is a spanning word-representable subgraph of $G_{1}$ and $G_{2}$, respectively.

By the definition of the corona product $H = G_1 \odot G_2$, for each vertex $v \in V(G_1)$, there is a copy of $G_2$ attached to $v$, denoted by $G_2^v$. Correspondingly, for each $j \in \{1, \ldots, k\}$, let $B_j^v$ denote the subgraph of $G_2^v$ isomorphic to $B_j$. The vertex set of $H$ is given by the disjoint union
\[
V(H) = V(G_1) \cup \bigcup_{v \in V(G_1)} V(G_2^v).
\]

We now show that $\mu(H) \leq k + 1$ by constructing a collection of $k+1$ spanning subgraphs of $H$, denoted by $A_1'$ and $\{D_i\}_{i=2}^{k+1}$, defined by their edge sets as follows:
\begin{align*}
E(A_1') &= E(A_1) \cup \{vu : v \in V(G_1), \, u \in V(G_2^v)\}, \\
E(D_i) &= E(A_i) \cup \bigcup_{v \in V(G_1)} E(B_{i-1}^v), \quad \text{for } 2 \leq i \leq k, \\
E(D_{k+1}) &= \bigcup_{v \in V(G_1)} E(B_k^v).
\end{align*}

\noindent We now verify the word-representability of each spanning subgraph in this collection:
\begin{enumerate}
      \item \textbf{Word-representability of $A_1'$:} Observe that the subgraph $A_1'$ is precisely the corona product $A_1 \odot (V(G_2), \emptyset)$, where $(V(G_2), \emptyset)$ denotes the empty graph on the vertex set of $G_2$. By construction, $A_1$ is a word-representable graph. Since every empty graph is a comparability graph, it follows directly from Theorem~\ref{corona-prod-chara-word} that $A_1'$ is word-representable.

    \item \textbf{Word-representability of $D_i$ ($2 \leq i \leq k$):} For each $i$, the graph $D_i$ is the union of the spanning subgraph $A_i$ and the collection of subgraphs $\{B_{i-1}^v\}_{v \in V(G_1)}$. Since $A_i$ and each graph in the collection $\{B_{i-1}^v\}_{v \in V(G_1)}$ are word-representable and have pairwise disjoint vertex sets, it follows from Theorem~\ref{unionofwrgiswrg} that $D_i$ is word-representable for all $2 \leq i \leq k$.

       \item \textbf{Word-representability of $D_{k+1}$:} The graph $D_{k+1}$ is the union of the collection of subgraphs $\{B_k^v\}_{v \in V(G_1)}$. Since these subgraphs are word-representable and have pairwise disjoint vertex sets, Theorem~\ref{unionofwrgiswrg} guarantees that $D_{k+1}$ is word-representable.

\end{enumerate}

Every edge of $H$ lies in at least one of the graphs $A_1', D_2, \ldots, D_{k+1}$. Consequently,
\[
\mu(H) \leq k + 1 = \max\{\mu(G_1), \mu(G_2)\} + 1.
\]

Now suppose that $k_1 > k_2$. Then $k=\max\{k_1,k_2\}=k_1$, and hence $B_k=(V(G_2),\emptyset)$. Consequently,
$
D_{k+1}
=
\bigcup_{v\in V(G_1)} B_k^v
$
is an empty graph, i.e., $E(D_{k+1})=\emptyset$. Hence, $D_{k+1}$ may be omitted from the covering. It follows that the graphs $
A_1', D_2, \ldots, D_k$
already form a covering of \(H\) by \(k\) word-representable graphs. Hence,
$
\mu(H)\le k=\max\{\mu(G_1),\mu(G_2)\}.
$

Combined with the lower bound
$
\mu(H)\ge \max\{\mu(G_1),\mu(G_2)\},
$
established earlier, we obtain
\[
\mu(H)=\max\{\mu(G_1),\mu(G_2)\}.
\]

Now suppose that $G_2$ admits a covering by $\mu(G_2)=k_2$ word-representable graphs in which at least one graph is a comparability graph. Without loss of generality, assume that in the representation $G_2=\bigcup_{j=1}^{k_2} B_j$, the graph $B_1$ is a comparability graph.

We now consider the previously defined families $\{A_i\}_{i=1}^{k}$ and $\{B_j\}_{j=1}^{k}$, where the extension by empty spanning subgraphs (if necessary) is as in the general upper bound construction. Using these, we construct a collection of $k$ spanning subgraphs $\{D_i'\}_{i=1}^{k}$ of $H$ by defining their edge sets as follows:
\begin{align*}
E(D_1')
&=
E(A_1)
\cup
\{vu : v\in V(G_1),\, u\in V(G_2^v)\}
\cup
\bigcup_{v\in V(G_1)} E(B_1^v),\\
E(D_i')
&=
E(A_i)
\cup
\bigcup_{v\in V(G_1)} E(B_i^v),
\qquad 2\le i\le k.
\end{align*}

We now verify that each graph $D_i'$ is word-representable.

\begin{enumerate}
\item \textbf{Word-representability of $D_1'$.}  
Observe that $D_1'$ is precisely the corona product $A_1\odot B_1$. Since $A_1$ is word-representable and $B_1$ is a comparability graph, Theorem~\ref{corona-prod-chara-word} implies that $D_1'$ is word-representable.

\item \textbf{Word-representability of $D_i'$ for $2\le i\le k$.}  
For each $i\ge 2$, the graph $D_i'$ is the union of the spanning subgraph $A_i$ and the collection of subgraphs $\{B_i^v\}_{v\in V(G_1)}$. Since $A_i$ and each $B_i^v$ are word-representable and their vertex sets are pairwise disjoint, it follows directly from Theorem~\ref{unionofwrgiswrg} that $D_i'$ is word-representable.
\end{enumerate}

The union of the edge sets of $D_1',D_2',\ldots,D_k'$ is precisely $E(H)$. Hence, $\{D_i'\}_{i=1}^{k}$ is a covering of $H$ by $k$ word-representable graphs, and therefore $
\mu(H)\le k=\max\{\mu(G_1),\mu(G_2)\}.
$ Together with the lower bound
$
\mu(H)\ge \max\{\mu(G_1),\mu(G_2)\},
$
established earlier, we conclude that
\[
\mu(H)=\max\{\mu(G_1),\mu(G_2)\}.
\]
\end{proof}

\begin{remark}
The case $\mu(G_1)=\mu(G_2)=1$ where $G_2$ is not a comparability graph demonstrates that the upper bound in Theorem~\ref{corona-gen-thm} is attainable. Indeed, under these conditions, Theorem~\ref{corona-prod-chara-word} implies that $H$ is not word-representable, so $\mu(H) \geq 2$. Since Theorem~\ref{corona-gen-thm} establishes $\mu(H) \leq 2$, it follows that $\mu(H) = \max\{\mu(G_1),\mu(G_2)\}+1 = 2$.
\end{remark}

The structural characterization established in Theorem~\ref{corona-gen-thm} leaves an unresolved regime specifically when $\mu(G_1) \le \mu(G_2)$ (with $\mu(G_2) \ge 2$) and $G_2$ does not satisfy the specified covering condition. Motivated by the fact that satisfying this condition ensures the upper bound tightens to exactly $\max\{\mu(G_1), \mu(G_2)\}$, we pose the following question.

\begin{question}\label{q:comparability_covering}
Let $G$ be a graph with $\mu(G) \ge 2$. Does there exist a covering of $G$ by $\mu(G)$ word-representable subgraphs such that at least one subgraph in the covering is a comparability graph?
\end{question}

An affirmative answer to Question~\ref{q:comparability_covering} would imply that the hypothesis of Theorem~\ref{corona-gen-thm} is satisfied for every graph $G_2$ with $\mu(G_2) \ge 2$. Furthermore, if $\mu(G_2) = 1$ and $G_2$ is a comparability graph, the condition is trivially satisfied. Under an affirmative answer, the only scenario where the multi-word-representation number of the corona product does not perfectly collapse to $\max\{\mu(G_1), \mu(G_2)\}$ would be the exact gap identified in our Remark: when $\mu(G_1) \le 1$, $\mu(G_2) = 1$, and $G_2$ is not a comparability graph (yielding $\mu(H) = 2$). Beyond its direct application to the corona product, Question~\ref{q:comparability_covering} is of independent interest in structural graph theory.

\section{Tensor product of graphs and multi-word-representation number}
\label{sec:tensor}

For the tensor product of graphs, the exact characterization of word-representability remains an open problem. In~\cite{choi2019operations}, it was shown that the tensor product of two word-representable graphs is not necessarily word-representable; a concrete counterexample was provided.

It is well known that for any graphs $G_{1}$ and $G_{2}$, 
$\chi(G_{1} \times G_{2}) \le \min \{\chi(G_{1}), \chi(G_{2})\}$ \cite{Hedetniemi1967HomomorphismsOG}.
Moreover, for any graph $G$, Theorem~\ref{mu-log-chi-bound} implies that $\mu(G) \le \log_{3} \chi(G)$. Combining these two facts, we obtain the following upper bound on the multi-word-representation number of the tensor product.

\begin{observation}
\label{tensor-prod-multi-wordrep-bound}
Let $G_{1}$ and $G_{2}$ be graphs. Then
\[
\mu(G_{1} \times G_{2}) \le \log_{3} \bigl( \min \{\chi(G_{1}), \chi(G_{2})\} \bigr).
\]
\end{observation}

\begin{remark}
Regarding Observation~\ref{tensor-prod-multi-wordrep-bound}, this upper bound cannot be improved to $\mu(G_{1} \times G_{2}) \le \min \{\mu(G_{1}), \mu(G_{2})\}$, nor does a bound of the form $\mu(G_{1} \times G_{2}) \le \max \{\mu(G_{1}), \mu(G_{2})\}$ hold. Indeed, the counterexample in~\cite{choi2019operations} provides word-representable graphs $G_{1}$ and $G_{2}$ (where $\mu(G_{1}) = \mu(G_{2}) = 1$) whose tensor product $H$ satisfies $\mu(H) \ge 2$ (as $H$ is non-word-representable).

Conversely, Observation~\ref{tensor-prod-multi-wordrep-bound} immediately yields a notable sufficient condition for word-representability. If at least one of the graphs ($G_{1}$ or $G_{2}$) is $3$-colorable, then $\min \{\chi(G_{1}), \chi(G_{2})\} \le 3$. By our bound, this implies $\mu(G_{1} \times G_{2}) \le \log_{3}(3) = 1$. Since $\mu$ is a positive integer, we have $\mu(G_{1} \times G_{2}) = 1$, meaning the tensor product is guaranteed to be word-representable. This provides a direct, partial answer to Problem~7.2.5 posed by Kitaev in~\cite[Chapter~7]{book}, which concerns the word-representability of tensor products. Specifically, our result demonstrates that as long as one of the graphs is $3$-colorable, the resulting product graph is unconditionally word-representable, thereby completely settling the outcome regardless of the second graph's word-representability status.
\end{remark}

\section{Strong product of graphs and multi-word-representation number}
\label{sec:strong}

We now turn to the strong product. The exact characterization of word-representability remains an open problem for strong products as well. Because the edge set of the strong product $G_1 \boxtimes G_2$ is the union of the edge sets of the Cartesian product $G_1 \square G_2$ and the tensor product $G_1 \times G_2$, we can leverage our previous results. By combining Definition~\ref{def-strong-product}, Theorem~\ref{thm:cartesian-multi-rep}, and Observation~\ref{tensor-prod-multi-wordrep-bound}, we obtain the following bounds on the multi-word-representation number of the strong product.

\begin{observation}
Let $G_{1}$ and $G_{2}$ be graphs. Then
\[
\mu(G_{1} \boxtimes G_{2}) \le \max \{\mu(G_{1}), \mu(G_{2})\} + \log_{3} \left( \min \{\chi(G_{1}), \chi(G_{2})\} \right).
\]
\end{observation}

For the corresponding lower bound, observe that for any fixed vertex $v \in V(G_2)$, the subgraph of $G_1 \boxtimes G_2$ induced by $\{(u,v) : u \in V(G_1)\}$ is isomorphic to $G_1$. Similarly, for any fixed vertex $u \in V(G_1)$, the subgraph induced by $\{(u,v) : v \in V(G_2)\}$ is isomorphic to $G_2$. Thus, both $G_1$ and $G_2$ occur as induced subgraphs of $G_1 \boxtimes G_2$. Hence, by Remark~\ref{monotonic-mu-hereditary-remark}, we obtain the following:

\begin{observation}
\label{strong-prod-lower-bound}
Let $G_1$ and $G_2$ be graphs. Then
\[
\mu(G_1 \boxtimes G_2) \ge \max \{\mu(G_1), \mu(G_2)\}.
\]
\end{observation}

\begin{remark}
The tensor product $G_{1} \times G_{2}$ does not, in general, contain either $G_{1}$ or $G_{2}$ as an induced subgraph. Consequently, the structural argument used to obtain the lower bound in the case of the strong product (Observation~\ref{strong-prod-lower-bound}) does not apply to the tensor product, and an analogous general lower bound cannot be derived by the same method.
\end{remark}

\section{Concluding remarks}
\label{sec:conclusion}

In this paper, we investigated the multi-word-representation number, $\mu(G)$, for graphs generated by six fundamental graph products: the Cartesian, rooted, lexicographic, corona, tensor, and strong products. We determined the exact value of $\mu$ for the Cartesian and rooted products. For the lexicographic and corona products, we established general upper and lower bounds and identified sufficient conditions under which these bounds collapse to exact evaluations. Notably, we proved that for the lexicographic product of any two minimal non-word-representable graphs, the multi-word-representation number is at most $3$. For the tensor and strong products, we established logarithmic bounds in terms of the chromatic numbers of the factor graphs, and observed that the tensor product of two graphs is word-representable if either of the factors is $3$-colorable.

Furthermore, we explored the multi-word-representation number for the lexicographic powers of a graph. For any integer $k \ge 2$, we established upper bounds for $\mu(G^{[k]})$ and characterized the word-representability of $G^{[k]}$. As a primary application, we utilized these lexicographic powers to investigate the extremal function $\tau(n)$, which denotes the largest integer such that every graph on $n$ vertices is guaranteed to contain a word-representable induced subgraph of at least that size. While $\tau(n)$ was introduced and evaluated only for small orders $n \le 9$ in~\cite{multi-word-rep}, this work provides the first general bounds on its asymptotic behavior. By constructing a family of extremal graphs via the $k$-th lexicographic power of a fixed base graph $H$ (Figure~\ref{fig:tau8neq7}) on $8$ vertices, we proved that the maximum size of a word-representable induced subgraph in this $8^k$-vertex family is at most $6^k$. This explicit construction yields the upper bound $\tau(n) \le n^{\log_8 6 + \epsilon}$ for any $\epsilon > 0$, formally demonstrating that $\tau(n)$ grows sublinearly. This construction also highlights a sharp structural contrast between the two underlying invariants: for all $k \ge 1$, the graph $H^{[k]}$ maintains a strictly bounded multi-word-representation number of $\mu(H^{[k]}) = 2$, even though the largest word-representable induced subgraph becomes asymptotically negligible relative to the order of the entire graph. Additionally, when restricted to the class of perfect graphs, we established a polynomial lower bound for the function $\tau(n)$. Building on these asymptotic results, we proved that for every integer $n \geq 433$, there exists a graph of order $n$ that does not admit a word-representable bipartition, thereby providing a negative answer to the Word-representable Bipartition (WB) problem for all such $n$. While the WB problem is settled for $n \le 13$  and $n \ge 433$, the status of the problem for the range $14 \le n \le 432$ remains open.

Finally, our findings motivate four specific questions for future investigation:
\begin{itemize}
    \item \textbf{Question~\ref{q:lex_product_gap}} seeks to determine the exact multi-word-representation number of the lexicographic product in the unresolved regime where $\mathrm{cov}_{\mathrm{comp}}(G_2) > \max\{\mu(G_1), \mu(G_2)\}$.
    
    \item \textbf{Questions~\ref{q:minimal_non_wr_cover} and~\ref{q:tightness_mu_3}} focus on minimal non-word-representable graphs. Specifically, Question~\ref{q:minimal_non_wr_cover} asks whether every minimal non-word-representable graph admits a cover number by comparability graphs of at most $2$. While it is known that there exist word-representable graphs with a cover number by comparability graphs of $\Omega(\log n)$~\cite{kenkireth2026word}, determining whether the strict condition of minimality inherently suppresses this parameter remains an intriguing open problem. 
    
    \item Closely tied to the above is \textbf{Question~\ref{q:tightness_mu_3}}, which asks whether there exist minimal non-word-representable graphs $G_1$ and $G_2$ that force the lexicographic product to attain the established upper bound of $\mu(G_1 \circ G_2) = 3$. 
    
   \item Lastly, \textbf{Question~\ref{q:comparability_covering}} asks whether every graph $G$ with $\mu(G) \ge 2$ admits a covering by $\mu(G)$ word-representable subgraphs such that at least one subgraph is a comparability graph. An affirmative answer to this would imply that the general upper bound for the corona product inherently collapses to the exact value $\max\{\mu(G_1), \mu(G_2)\}$ in all cases where $\mu(G_2) \ge 2$.
\end{itemize}

\bibliographystyle{plain}  
\bibliography{references}

\end{document}